\documentclass[11pt,a4paper]{amsart}
\usepackage[utf8]{inputenc}
\usepackage{amsmath}
\usepackage{amssymb}
\usepackage{hyperref}
\usepackage[a4paper,bindingoffset=0in,left=2cm,right=2cm,top=3.5cm,bottom=3.5cm,footskip=.25in]{geometry}
\usepackage{amsthm}
\usepackage{tikz-cd}
\usepackage{cleveref}
\usepackage{xypic}
\usepackage[shortlabels]{enumitem}
\usepackage{MnSymbol}

\theoremstyle{definition}
\newtheorem{thm}{Theorem}[section]
\newtheorem{prop}[thm]{Proposition}
\newtheorem{cor}[thm]{Corollary}
\newtheorem{lem}[thm]{Lemma}

\newtheorem{defi}[thm]{Definition}
\newtheorem{rem}[thm]{Remark}
\newtheorem{ex}[thm]{Example}

\newtheorem{thmintro}{Theorem}

\newcommand{\Q}{\mathbb{Q}}
\newcommand{\QQ}{\mathbb{Q}}

\newcommand{\Z}{\mathbb{Z}}
\newcommand{\Id}{\operatorname{Id}}

\newcommand{\id}{\operatorname{id}}

\title{Poincar\'e dualization and Massey products}

\begin{document}
    \author{Aleksandar Milivojevi\'c}
	\address{University of Waterloo, Faculty of Mathematics, 200 University Avenue West, N2L 3G1, Waterloo, Ontario }
	\email{amilivoj@uwaterloo.ca}
	\author{Jonas Stelzig}
	\address{ Mathematisches Institut der Ludwig-Maximilians-Universit\"at M\"unchen,
		Theresienstraße 39, 80993 M\"unchen}
  \email{jonas.stelzig@math.lmu.de}
	\author{Leopold Zoller}
    \address{Universitat de Barcelona, Facultat de Matemàtiques i Informàtica, Gran Via de les Corts Catalanes, 585, 08007 Barcelona}
    \email{leopold.zoller@ub.edu}
	
	\begin{abstract} We study the rational homotopy theoretic and geometric properties of a construction which extends any cohomologically connected, finite type cdga to one satisfying cohomological Poincar\'e duality. Using this construction we show that non-trivial quadruple Massey products can pull back trivially under non-zero degree maps of Poincar\'e duality spaces, unlike the case of triple Massey products as studied by Taylor. We also show that a non-zero degree map between formal rational Poincar\'e duality spaces need not be formal. Our consideration of Massey products naturally ties in with cyclic $A_\infty$-algebras modelling Poincar\'e duality spaces.
	\end{abstract}

\maketitle

\setcounter{tocdepth}{1}
\tableofcontents

\section{Introduction}

We treat the following three topics concerning rational Poincar\'e duality and formality: \begin{itemize} \item We study a construction which extends any cohomologically connected, finite type cdga by its dual to one satisfying  Poincar\'e duality on its cohomology, which geometrically corresponds to taking the double of a thickened representing cell complex. \item Using this construction, we build examples of dominant maps to Poincar\'e duality spaces showing that the main result of \cite{MSZ23}, namely that formality is preserved under dominant maps, is sharp in a certain sense. \item We are led to revisit nice algebraic (namely, cyclic) models of Poincar\'e duality spaces, using which we can give quick proofs of known formality results. \end{itemize}

The existence of non-zero degree maps between closed manifolds, giving a relation going by the name of \emph{domination}, has been of substantial interest going back at least to work of Gromov, Milnor, and Thurston in the 1970's, see \cite[p.173]{CT89}. The general empirical observation is that the domain of a non-zero degree map should be ``at least as complicated'' as its target, see e.g. loc. cit. As an instance of this heuristic, the authors showed in \cite{MSZ23} that formality is preserved under domination. This observation was in part motivated by a result of Taylor that non-trivial triple Massey products remain non-trivial upon pullback by a dominant map.
    
One is then naturally led to ask about the behavior of quadruple and higher Massey products under pullback by dominant maps.  In order to address this problem, we detail a construction that extends any cdga to one satisfying Poincar\'e duality on its cohomology, and discuss its functoriality. The construction is simple and by no means new, but we wish to give it a down-to-earth, workable description, and study its rational homotopy theoretic properties, allowing us to easily construct examples by hand. Explicitly we have:

 \begin{thmintro}\label{thmA} For every $n$ and any cohomologically connected, finite type cdga $A$, cohomologically concentrated in degrees $<n$, the square-zero extension by its shifted dual $P_nA:=A\oplus D_nA$ is a new cdga satisfying $n$-dimensional Poincar\'e duality on its cohomology.  This construction has the following properties:
 
 \begin{enumerate}
\item The cdga $P_n A$ admits a cdga retract back to $A$. If $A \to B$ is a morphism of cdga's which admits a retract $B \to A$ of dg-$A$-modules, then $A \to B$ extends to a non-zero degree map $P_n A \to P_n B$. 
 \item The cdga $A$ is formal if and only if $P_n A$ is formal (\Cref{prop:PDnotFormal}, \Cref{cor:qi-invariance}). \item A Massey product is non-trivial on $A$ if and only if it is non-trivial upon inclusion into $P_n A$ (\Cref{prop:PDnotFormal}). \item If $A$ models the rational homotopy type of a finite complex $X$ embedded in Euclidean $n$-space, then $P_n A$ models the double of a thickening of $X$ to a manifold with boundary (\Cref{prop:topological interpretation}). \end{enumerate}
 \end{thmintro}
 Making use of this \emph{Poincar\'e dualization} construction, we show:

 \begin{thmintro}\label{thmB}(\Cref{thm: counterex}) There exists a non-zero degree map $Y \to X$ between rational Poincar\'e duality spaces, where $X$ carries a non-trivial quadruple Massey product whose pullback to $Y$ is trivial.
 \end{thmintro}

 Hence the above-mentioned result by Taylor on triple Massey products does not extend to quadruple products. This is furthermore to be contrasted with the main theorem of \cite{MSZ23}, i.e.\ formality being preserved by non-zero degree maps, which in a sense does generalize Taylor's theorem to higher Massey products as long as vanishing of Massey products is understood in a suitably uniform sense. Incidentally, the example constructed for \Cref{thmB} also shows that a cdga may have no rational Massey products, while its Poincar\'e dualization does have nontrivial Massey products, \Cref{homotopicalpropertiesofB}.
 
 With formality being inherited under a non-zero degree map it is natural to ask whether in this case the map itself is formal. Again drawing upon the Poincaré dualization construction and its naturality properties we find that this is not the case, proving:
 
 \begin{thmintro}\label{thmC}(\Cref{nonzerodegnonformal}) A non-zero degree map between formal rational Poincar\'e duality spaces need not be formal. \end{thmintro}

 In order to construct the example for the above theorem we extend the Poincar\'e dualization construction to the category of $A_\infty$-algebras (\Cref{inftyPDsection}). We find that the defining formula for the dualization appears naturally for any minimal $C_\infty$-model of a Poincaré duality space provided there are no higher operadic Massey products landing in top degree (\Cref{prop:formulainsimplyPDmodel}). This is reminiscent of how ad hoc Massey products landing in top degree on a Poincar\'e duality cdga vanish. One is led to the notion of cyclic $A_\infty$-algebras (\Cref{simplePDmodels}). Using cyclic models one can relatively quickly recover some well-known results in rational homotopy theory on formality of Poincar\'e duality spaces given some connectivity assumption (\Cref{simplePDmodels}). 
 
 In \Cref{Preliminaries} we review Massey products, recover Taylor's theorem, and illustrate the necessity of Poincar\'e duality therein. Then in \Cref{poincaredualization} we detail the Poincar\'e dualization construction for cdga's, and note that a cdga morphism $A \to B$ extends to a non-zero degree map $P_n A \to P_n B$ when there is a dg-$A$-module retract $B \to A$. In \Cref{counterexample} we construct a cdga morphism $A \to B$ admitting a dg-$A$-module retract such that a non-trivial quadruple Massey product on $A$ (and hence $P_n A$) becomes trivial on $B$, proving \Cref{thmB}. We further study the Massey products on $B$ in \Cref{homotopicalpropertiesofB}; it is a cohomologically simply connected six-dimensional non-formal cdga with no non-trivial Massey products. In \Cref{inftyPDsection} we extend the discussion of \Cref{poincaredualization} to the $A_\infty$ setting. In \Cref{nonformalmapsection} we construct an example showing \Cref{thmC}, and in \Cref{universalpropertyandgeometricinterpretation} we complete the proof of \Cref{thmA} by giving the geometric interpretation of the algebraic dualization construction, extending results of Lambrechts \cite{La00}. In \Cref{simplePDmodels} we discuss cyclic models of Poincar\'e duality spaces and use them to recover results of Miller, Cavalcanti, and  prove cases of a conjecture of Zhou.
 
\subsection*{Acknowledgements} 

The authors would like to thank Michael Albanese, Manuel Amann, Joana Cirici, Vladimir Dotsenko, Pavel H\'ajek, Dieter Kotschick,  John Morgan, Christoforos Neofytidis, Dennis Sullivan, Peter Teichner, Nathalie Wahl, Shmuel Weinberger, and Scott Wilson for illuminating conversations and helpful comments.

Part of this research was supported through a ``Research in Pairs'' program visit by A.M. and J.S. at the Mathematisches Forschungsinstitut Oberwolfach; they thank the MFO for the excellent working conditions provided there. They are likewise grateful to the City University of New York Graduate Center for their hospitality. A.M. would also thank the Institut Mittag-Leffler in Djursholm for its generous hospitality during a visit to the ``Higher algebraic structures in algebra, topology and geometry'' program; he thanks \"Ozg\"ur Bay\i nd\i r, Alexander Berglund, and Andrea Bianchi for interesting discussions there.

J.S. and L.Z. thank the Ludwig-Maximilians-Universit\"at M\"unchen, and A.M. thanks the Max-Planck-Institut f\"ur Mathematik, for their enduring support.

We also thank the referee for a thorough reading and many helpful remarks which improved the paper.

\section{Preliminaries}\label{Preliminaries}

We recall the concepts used in the introduction and throughout, and set notation. We will consider graded commutative algebras $(A,d)$ over $\QQ$ (though much of the purely algebraic discussion goes through for an arbitrary field), where we allow entries in negative degrees, i.e. $A=\bigoplus_{k\in\Z}A^k$, but most of the time we restrict to \emph{cohomologically connected} cdga's, i.e. those where $H^k(A)=0$ for $k<0$ and  $H^0(A)\cong \QQ$. If $A^k=0$ for $k<0$ and $A^0=\QQ$, we say $A$ is connected. We call $A$ a \emph{(rational) Poincar\'e duality cdga} if its cohomology satisfies Poincar\'e duality. That is, there is an index $n$ such that $H^n(A,d) \cong \QQ$ and the pairing $$H^k(A,d)\otimes H^{n-k}(A,d) \to H^n(A,d) \cong \QQ$$ given by $$\alpha \otimes \beta \mapsto \alpha \beta$$ is non-degenerate.

A rational commutative differential graded algebra (cdga) is said to be \emph{formal} if there is a zigzag of quasi-isomorphisms of cdga's

$$
\begin{tikzcd}[row sep=small]
        & {(B_1,d)} \arrow[ld] \arrow[rd] &           & \cdots \arrow[ld] \arrow[rd] &           & {(H,0)} \arrow[ld] \\
{(A,d)} &                                 & {(B_2,d)} &                              & {(B_r,d)} &                   
\end{tikzcd}
$$
connecting $(A,d)$ to a cdga with trivial differential. For $A$ cohomologically connected,  one may pick a Sullivan model $(\Lambda V,d)\to A$, i.e. a connected cdga that is free as an algebra, satisfying a nilpotence condition (c.f. \cite{FHT12}), with a quasi-isomorphism to $A$. One may even pick $(\Lambda V,d)$ to be minimal, i.e. $d(\Lambda V)\subseteq \Lambda^{\geq 2}V$;
\cite[p.191]{FHT12}.  In terms of such a model, formality of $A$ is equivalent to the existence of a quasi-isomorphism $(\Lambda V,d)\to (H(A),0)$. (That is, we may replace the chain of quasi-isomorphisms by a single ``roof''.) 

Computable obstructions to formality are given by (ad hoc) Massey products \cite[Section 2]{M58}. Given three pure-degree classes $[x],[y],[z]\in H(A)$ such that $xy=da$, $yz=db$, the element $az-(-1)^{|x|}xb$ is closed and therefore gives rise to a cohomology class. Modulo the ideal generated by $[x]$ and $[y]$, this class is well-defined and independent of the choices of representatives and primitives. It is called the triple Massey product and denoted $\langle [x],[y],[z] \rangle :=[az-(-1)^{|x|}xb]\in H(A)/([x],[y])$. Equivalently, the triple Massey product is the set of classes $\{[az - (-1)^{|x|}xb]\}$ obtained for all choices of primitives $a,b$. 

Quadruple Massey products are defined similarly: Given four classes $[w],[x],[y],[z]\in H(A)$, which for simplicity we assume to have pure even degree (which will be the case for us below), a defining system for the quadruple product $\langle [w],[x],[y],[z]\rangle\subseteq H(A)$ is a collection of pure-degree elements $a,b,c,f,g$ such that $da=wx$, $db=xy$, $dc=yz$ and $df=ay-wb$ and $dg=bz-xc$. For any such defining system, one obtains a cohomology class $[wg+ac+zf]\in H(A)$. The quadruple Massey product $\langle [w],[x],[y],[z]\rangle\subseteq H(A)$ is then defined to be the collection of classes obtained from all such defining systems. Again, this collection is independent of the chosen representatives for the classes. As in the case of the triple product, the quadruple product is said to be trivial (or vanish) if $0\in\langle [w],[x],[y],[z]\rangle$. The definitions for quintuple and higher products are similar; we refer the reader to \cite{K66}.

Massey products are invariants of the quasi-isomorphism type of a cdga, and on formal cdga's all Massey products vanish.

To a topological space $X$ we can associate its connected cdga $A_{PL}(X)$ of rational piecewise-linear forms \cite{Su77}, \cite{DGMS75}; this cdga computes the rational cohomology of $X$ (see e.g. \cite[Theorem 2.1]{DGMS75}, \cite[Theorem 1.21]{H07}). We say the space $X$ is formal if $A_{PL}(X)$ is formal as a cdga \cite[p.260]{DGMS75}, \cite[Definition 2.1]{H07}. A space is a \emph{rational Poincar\'e duality space} if its rational cohomology satisfies Poincar\'e duality.
 
 Let us give an alternative proof of Taylor's theorem \cite{Ta10} mentioned above.
	
	\begin{prop}\label{prop:Taylor} Let $Y\rightarrow X$ be a non-zero degree map between rational Poincar\'e duality spaces and let $a,b,c\in H(X)$ with $ab=bc=0$. If \[m:=\langle a,b,c\rangle\neq 0\in \frac{H(X)}{a\cup H(X)+H(X)\cup c},\]
		then also \[
		f^*(m)\neq 0\in \frac{H(Y)}{f^*a\cup H(Y)+H(Y)\cup f^*c}.
		\]
	\end{prop}
	
	\begin{proof}
		The map $f^*: H(X)\to H(Y)$ has a one-sided inverse given by $f_*':=\frac{1}{\deg f} f_*$, i.e. $f_*'f^*=Id_{H(X)}$. Here $f_*$ denotes the pushforward, determined by $f^*$ and Poincar\'e duality. This yields a splitting
		\[
		H(Y)\overset{(f_*',\operatorname{pr})}{\longrightarrow} H(X)\oplus H(Y)/f^*H(X).
		\]
		By the projection formula
		\[
		f_*(f^*x\cup y)=x\cup f_*y \quad\text{for }x\in H(X), y\in H(Y),\]
		this (additive) splitting is compatible with the natural $H(X)$-module structures on both sides (given by $f^*$ on the left and $(Id, f^*)$ on the right). Therefore, writing $K:=H(Y)/f^*H(X)$, the domain of definition of the Massey product decomposes as
		
		\[
		\frac{H(Y)}{f^*a\cup H(Y)+H(Y)\cup f^*c}\overset{\sim}{\longrightarrow}\frac{H(X)}{a\cup H(X)+H(X)\cup c}\oplus\frac{K}{f^*a \cup K+K\cup f^*c},
		\]
		and under this splitting, we have $f^*(m)=(m,0)$.
	\end{proof}
	
	\begin{rem} The above proof, like Taylor's, works for maps of Poincar\'e duality cdga's over any field, as long as we interpret $\deg f\neq 0$ to mean that $\deg f$ is invertible. 
\end{rem}

\begin{rem}
    One can weaken the hypothesis and only require $X$ to satisfy rational Poincare duality and the map $f^ *:H(X)\to H(Y)$ to be injective (which is equivalent to $f^*$ being nonzero in the top cohomology of $X$). In fact, let $\int_X\in \hom(H(X),\Q)=:H(X)^\vee$ be the element corresponding to the isomorphism $H^n(X)\cong \Q$ in degree $n$ and zero in all other degrees.  Let $\int_Y\in \hom(H(Y),\Q)$ be a preimage of $\int_X$ under $(f^*)^\vee$. Now, define $f_*':H(Y)\to H(X)^\vee\cong H(X)$ by $\beta\mapsto (\alpha\mapsto \int_Yf^ *\alpha\beta)$, followed by Poincar\'e duality on $X$. The above proof still applies.
\end{rem}
	
	\begin{ex}\label{noPD} One cannot drop the Poincar\'e duality assumption on $X$. In fact, it is easy to find examples of cohomologically injective maps of cdga's such that a non-vanishing triple Massey product in the domain vanishes in the target. For example, consider the inclusion of cdga's $$A:=\left( \Lambda(x,y,z), dz = xy \right)\hookrightarrow B:=\left( \Lambda(x,y,z,u,v), dz = xy, dv = xz-yu \right),$$ where all generators are in degree 1. This induces an inclusion $A':=A/A^{\geq 3}\hookrightarrow B':=B/B^{\geq 3}$ which is injective on cohomology. Now, $\langle x, x, y \rangle$ is a non-vanishing triple product in $A'$, while in $B'$ it is represented by $[xz] = [yu]$, which lies in the indeterminacy.\end{ex}

\begin{ex}\label{noPDformal} Continuing along the lines of \Cref{noPD}, we give an example of a non-formal cdga with a map to a formal cdga which is injective on cohomology. Namely, take $$\left( \Lambda(X_2, Y_2, a_3, b_3, c_3), dX=dY = 0, da = X^2, db=XY, dc=Y^2 \right)$$ and $$\left( \Lambda(x_2, y_2, \alpha_3, \beta_3, \gamma_3)/(x^2, xy, y^2), d \equiv 0 \right),$$ where the indices on generators denote degrees. The map sending $X \mapsto x, Y \mapsto y, a \mapsto \alpha, b \mapsto \beta, c \mapsto \gamma$ descends to the truncation of both cdga's whereby we mod out the (differential) ideals of all elements of degree $\geq 6$. The resulting map is a cohomologically injective map from a non-formal cdga to a formal one; indeed, the domain carries the non-trivial triple Massey products $\langle [X], [X], [Y] \rangle$ and $\langle [X], [Y], [Y] \rangle$. \end{ex}

\section{Poincar\'e dualization}\label{poincaredualization}

We detail a construction that ``completes'' any cohomologically connected cdga to one satisfying Poincar\'e duality on its rational cohomology, which in certain cases is functorial. This construction is rather simple and has appeared before, see \cite{La00}, along with \cite[Proposition 14ff.]{KTV21}, \cite{LeV22} for a more recent context. Here we study it in detail within the context of rational homotopy theory.

Fix a natural number $n$. Let $(A,d)$ be a complex of rational vector spaces. We define the ($n$-th) \emph{dual complex} $D_nA$ by $(D_nA)^k:=(A^{n-k})^\vee$ with differential $(D_nA)^k \to (D_nA)^{k+1}$ given on pure-degree elements $\varphi\in D_n A$ by $d(\varphi)(a):=(-1)^{|\varphi|-1}\varphi(da)$ for any $a\in A$. Clearly, $D_n$ is a contravariant functor (given by $\left((D_n r)(\varphi)\right) (b) = \varphi(r(b))$ for a map of complexes $B \xrightarrow{r} A$) and
\begin{equation*}
    H^k(D_nA)=(H^{n-k}(A))^\vee.
\end{equation*}

We will be interested in the case where $A$ carries in addition the structure of a graded-commutative algebra, such that $d$ is a derivation, i.e. $(A,\wedge, d)$ is a cdga. It will not be necessary to assume in this definition that $A$ has only elements in non-negative degrees, and in fact this will in general not be the case for $P_n A$ defined below (but we will only care about the case where at least the cohomology of $A$ is concentrated in non-negative degrees). 

\begin{defi}
Let $A$ be a cdga. The $n$-th \textbf{Poincar\'e dualization} of $A$ is given, as a complex, by
\[
P_nA:= A\oplus D_nA,
\]
with multiplication (extending that on $A$) defined on pure-degree elements $a\in A$, $\varphi\in D_nA$ by the dual complex element given by
\begin{align*}
(a\wedge\varphi)(b)&:=(-1)^{|a||\varphi|}\varphi(a\wedge b),\\
(\varphi\wedge a)(b)&:=\varphi(a\wedge b),
\end{align*}
and setting $\varphi\wedge\psi=0$ for $\varphi,\psi\in D_n A$. 
\end{defi}

\begin{lem}
For a cdga $A$, the Poincar\'e dualization $P_nA$ is a cdga. 
\end{lem}
\begin{proof}
Graded commutativity of the multiplication holds by definition. For associativity, we only need to check the case $\varphi\in D_n A$ and $a,b \in A$ as all other combinations of products of three elements are either zero or entirely in $A$, where associativity holds since $A$ is a cdga. We compute, for $c \in A$:
\begin{align*}
    ((\varphi\wedge a)\wedge b)(c)&=(\varphi\wedge a)(b\wedge c)\\
    &=\varphi(a\wedge b\wedge c)\\
    &=(\varphi\wedge(a\wedge b))(c).
\end{align*}

That $d$ is a derivation again only has to be checked on products of the form $\varphi\wedge a$ with $\varphi\in D_n A$ and $a\in A$. In this case, we compute:
\begin{align*}
    (d\varphi\wedge a)(b)&=d\varphi(a\wedge b)\\
    &=(-1)^{|\varphi|-1}\varphi(d(a\wedge b))\\
    &=(-1)^{|\varphi|-1}\varphi(da\wedge b)+(-1)^{|\varphi|-1+|a|}\varphi(a\wedge db)\\
    &=(-1)^{|\varphi|-1}(\varphi\wedge da)(b) + (-1)^{|\varphi\wedge a|-1}(\varphi\wedge a)(db)\\
    &= (-1)^{|\varphi|-1}(\varphi\wedge da)(b)+d(\varphi\wedge a)(b). \qedhere
\end{align*}
\end{proof}


\begin{lem}\label{lem: nondegenerate}
Let $A$ be a cohomologically connected cdga such that $H(A)$ is finite dimensional and $H^k(A)=0$ for $k\geq n$. The cohomology $H(P_nA)$ is finite dimensional and concentrated in degrees $0,\ldots ,n$. Further, $P_nA$ is a Poincar\'e duality cdga, i.e. for any integer $k$, the pairing
\[
H^k(P_n A)\times H^{n-k}(P_nA)\overset{\wedge}{\longrightarrow} H^n(P_nA)\cong \QQ 
\]
is non-degenerate.
\end{lem}
\begin{proof}
By construction, $H^k(P_nA)=H^k(A) \oplus H^k(D_n A) \cong H^k(A) \oplus (H^{n-k}(A))^\vee$ and $H^{n-k}(P_nA)= H^{n-k}(A)\oplus H^{n-k}(D_nA)\cong H^{n-k}(A) \oplus (H^k(A))^\vee$,
and the pairing is given (up to a non-zero scalar) by evaluation. \qedhere
\end{proof}

\begin{ex}\label{examples} \,
\begin{itemize}
\item Let $A=\left( \Lambda(x)/x^2, d = 0 \right)$ with $|x|\geq 1$ and let $n>|x|$. Then \[P_n(A)\cong \left( \Lambda(x,y)/(x^2,y^2), d = 0 \right)\] with $|y|=n-|x|$, i.e. we obtain the cohomology algebra of the product of spheres $S^{|x|}\times S^{|y|}$.

\item Take $A$ to be a minimal model for $S^2$, i.e. $A = \left( \Lambda(x, y), dy = x^2 \right)$, which as a vector space has the basis $\{x^k, x^ky\}_{k\geq 0}$. Then $P_n(A)$ has a vector space basis given by the union of $\{x^k, x^ky\}$ and a dual basis $\{\widehat{x^k}, \widehat{x^k y}\}_{k\geq 0}$, the elements of which live in degrees $n - 2k$ and $n - 2k - 3$ respectively. These satisfy the following relations $$\widehat{x^k y} \wedge x^l y = \widehat{x^{k-l}}, \ \  \widehat{x^k y} \wedge x^l = \widehat{x^{k-l}y}, \ \  \widehat{x^k} \wedge x^l y = 0, \ \ \widehat{x^k} \wedge x^l = \widehat{x^{k-l}}$$ for $k\geq l$. All other products of dual elements with basis elements from $A$ are zero. The differential is determined by $d(\widehat{x})=d(\widehat{y})=0$ and for $k\geq 2$, $d(\widehat{x^k}) = - \widehat{x^{k-2}y}$, $d(\widehat{x^k y}) = 0$. One may check that this is connected to the first example (with $|x|=2$) by a chain of quasi-isomorphisms.
\item For a non-manifold example, consider the formal space $S^2 \vee S^3$, with model $$\left( \Lambda(x, y)/(x^2, xy), d = 0\right)$$ with $|x|=2$, $|y|=3$. Then for $n\geq 4$, its $n$-th Poincar\'e dualization is the cohomology ring of the manifold $\left( S^2 \times S^{n-2} \right) \# \left( S^3 \times S^{n-3} \right) $ equipped with trivial differential. Note that for large enough $n$, the boundary of a thickening of $S^2 \vee S^3$ in $\mathbb{R}^{n+1}$ is $\left( S^2 \times S^{n-2} \right) \# \left( S^3 \times S^{n-3} \right) $. \\ \end{itemize}

\end{ex}

\begin{rem}\label{geometric} We easily see that the Poincar\'e dualization of a Poincar\'e duality cdga is obtained by tensoring with $\QQ[x]/(x^2)$, with $x$ of the appropriate degree. Geometrically, this corresponds to crossing with a sphere (at least when $n$ is greater than the cohomological dimension of the original algebra). We discuss in \Cref{universalpropertyandgeometricinterpretation} how Poincar\'e dualization for a general cohomologically finite type cdga, at least for large enough $n$, corresponds to taking the double of a thickening of a corresponding cell complex embedded in Euclidean space (cf. the third example above).

 \end{rem}

The Poincar\'e duality cdga's that one can get via the Poincar\'e dualization construction are quite restricted. As a simple example, notice that $(\QQ[x]/(x^k), d=0)$, with $\deg(x) = 2$ and $k\geq 2$  (corresponding to $\mathbb{CP}^{\geq 2}$) cannot be obtained by Poincar\'e dualizing some cdga. The case of even $k$ also follows from the following:

\begin{prop} Let $(A,d)$ be a cdga as in \Cref{lem: nondegenerate}. For $n \equiv 0 \bmod 4$, the middle degree pairing $$H^{n/2}(P_n(A)) \otimes H^{n/2}(P_n(A)) \to \QQ$$ given by $\alpha \otimes \beta \mapsto c$, where $\alpha \beta = c \, \hat{1}$, is a sum of hyperbolic forms. In particular, the signature of this pairing is zero.
\end{prop}

\begin{proof} 

This follows from the equality $H^{n/2}(P_nA)=H^{n/2}(A)\oplus H^{n/2}(A)^\vee$. We can pick any basis for $H^{n/2}(A)$ and complete it with the dual basis to one for $H^{n/2}(P_nA)$. In this basis, the assertion is clear. \end{proof}


\begin{cor}\label{realization} If $(A,d)$ as in \Cref{lem: nondegenerate} is cohomologically simply connected, and $n$ (not necessarily divisible by four) is at least by two larger than the cohomological dimension of $A$, then $P_n(A)$ is realized by a simply connected closed smooth manifold $M$. That is, $A_{PL}(M)$ is connected by a chain of quasi-isomorphisms to $P_n(A)$. \end{cor}

As indicated in \Cref{geometric}, in \Cref{universalpropertyandgeometricinterpretation} we show that this manifold, for large enough $n$, can be taken to be the double of a thickening of a representing cell complex.

\begin{proof} By the above, if $n\equiv 0 \bmod 4$, the signature of $P_n(A)$ (with respect to any choice of generator for top degree rational homology) is zero. Notice also that respect to the dual of $\hat{1}$, the pairing is equivalent over the rationals to one of the form $\sum_{i=1}^s x_i^2 - \hat{x_i}^2$ (indeed, a rational change of basis takes the form $\left( \begin{smallmatrix} 0 & 1 \\ 1 & 0 \end{smallmatrix} \right)$ into $\left( \begin{smallmatrix} 1 & 0 \\ 0 & -1 \end{smallmatrix} \right)$). Hence by \cite[Theorem 13.2]{Su77}, for $n\geq 5$ (regardless of whether $n$ is divisible by four), choosing all rational Pontryagin classes to be trivial, there is a closed smooth manifold realizing this data. The only remaining non-trivial case is $n=4$, in which case our requirements force $A$ to have the cohomology ring of a wedge sum $\vee_\ell S^2$, and so $P_4(A)$ has the cohomology ring of the connected sum $\#_\ell (S^2 \times S^2)$, which is intrinsically formal, so it realizes $P_4(A)$.
\end{proof}



From now on, $A$, $B$ will denote cdga's satisfying the same finiteness and connectedness conditions as in \Cref{lem: nondegenerate}. 

For a map of cdga's $f:A\to B$ it is in general not true that it can be extended to a map $P_n A\to P_n B$. However, one has:
\begin{lem}\label{lem:naturality}
Given a map $r:B\to A$ of dg-$A$-modules, i.e.\ a map of complexes satisfying $r(f(a)\wedge b)=a\wedge r(b)$, with dual map $D_n r$, the map
\[
f\oplus D_n r: P_n A\to P_n B
\]
is a map of cdga's.  When $r(1) \neq 0$ (equivalently $r\circ f$ is a non-zero multiple of the identity), the map $f\oplus D_n r$ has non-zero degree.
\end{lem}
\begin{proof}
Because both $f: A\to B$ and $D_n r: D_n A \to D_n B$ are maps of complexes, so is $f\oplus D_n r$, where $f$ and $D_n r$ are extended trivially to all of $P_n A$. It thus remains to show that the map is compatible with the product. If both factors are in $A\subseteq P_n A$, this is true since $f$ is an algebra map. If both entries are in $D_n A$, their product is zero, and so is the product of their images under $D_n r$. The remaining case, $a\in A, \varphi\in D_n A$, follows from: 
\begin{align*}
    \Bigl( (f\oplus D_nr)(\varphi\wedge a) \Bigr) (b)&=\Bigl( D_n r(\varphi\wedge a) \Bigr) (b)\\
    &=(\varphi\wedge a)(r(b))\\
    &=\varphi(a\wedge r (b))\\
    &=\varphi(r(f(a)\wedge b))\\
    &= \left(D_nr(\varphi) \right) (f(a)\wedge b)\\
    &=\Bigl( D_nr(\varphi)\wedge f(a) \Bigr) (b)\\
    &=\Bigl( (f\oplus D_n r)(\varphi)\wedge(f\oplus D_n r)(a)\Bigr) (b).
\end{align*}
The statement about the degree follows since the top-degree cohomology in $P_n A$ and $P_n B$ is generated by any class that evaluates non-trivially on $1$.
\end{proof}
\begin{rem}
Consider the category whose objects are cdga's satisfying the conditions of $A$ as in \Cref{lem: nondegenerate} and whose morphisms $A\to B$ are pairs $(f,r)$ as in \Cref{lem:naturality}, with composition $(g,s)\circ (f,r)=(g\circ f, r\circ s)$. Then Poincar\'e dualization $P_n$ defines a functor from this category to that of Poincar\'e duality cdga's. Restricted to degree-wise finite dimensional cdga's, it is fully faithful.
\end{rem}

\begin{prop}\label{prop:PDnotFormal}
Let $(A,d)$ be a cdga, and let $i:A\to P_n A$ be the inclusion. If the Massey product $\langle x_1,...,x_m\rangle$ is non-zero, where $x_i \in H(A)$, then so is $\langle i(x_1),...,i(x_m)\rangle$.
\end{prop}

\begin{proof}
By construction, $P_nA=A\oplus D_n A$ with $A$ a subalgebra and $D_nA$ a differential ideal. Thus, the inclusion of cdga's $i:A\to P_nA$ admits a one-sided inverse map of cdga's  $r:P_n A\to A$ with $r\circ i=\id$. Now for any non-trivial Massey product $m\in H(A)$, $i(m)$ is non-trivial as $(r\circ i) (m)\subseteq m$. Recall, we treat a Massey product as the set of cohomology classes obtained via any possible defining system (see \cite{K66}), with the Massey product being trivial if the zero class is contained in this set.
\end{proof}

\begin{rem}\label{trivialMassey} Since we will implicitly use it in the next section, we remark on the following obvious property: if a Massey product $m$ is defined and trivial in $A$, then the same holds in $P_n(A)$ for any $n$, as $A$ embeds into $P_n(A)$. \end{rem}


Let us now discuss how formality of $A$ relates to the formality of $P_n A$.

\begin{prop}\label{formaldualization} If $P_n(A)$ is formal, then $A$ is formal. \end{prop}

\begin{proof} We mimic the proof that the retract of a formal space is formal \cite[Example 2.88]{FOT08}. Consider the maps of cdga's $A \xrightarrow{i} P_n(A) \xrightarrow{r} A$ used in the proof of \Cref{prop:PDnotFormal}. 
Take minimal models of $A$ and $P_n(A)$, and a map $\phi$ from the minimal model of $P_n(A)$ to its cohomology which induces the identity on cohomology \cite[Theorem 4.1]{DGMS75}. We have the following homotopy commutative diagram, where $\widehat{\imath}$ and $\widehat{r}$ denote the induced maps on minimal models:

$$\begin{tikzcd}
A \arrow[r, "i", hook]                      & P_n(A) \arrow[r, "r"]                                              & A                      \\
M(A) \arrow[u, "\sim"] \arrow[r, "\widehat{\imath}"] & M(P_n(A)) \arrow[u, "\sim"] \arrow[r, "\widehat{r}"] \arrow[d, "\phi"] & M(A) \arrow[u, "\sim"] \\
                                            & H(P_n(A)) \arrow[r, "{\widehat{r}}_*"]                                 & H(A)                  
\end{tikzcd}$$

Now the composition $\widehat{r}_* \phi \widehat{\imath}$ induces $\widehat{r}_* \widehat{\imath}_*$ on cohomology, which is an isomorphism since $r i = \mathrm{id}$. \end{proof}

We also have the converse, whose proof we postpone to \Cref{inftyPDsection}, as our proof will go through the category of $A_\infty$-algebras.

\begin{prop}[\Cref{AformaliffPAformal}] $A$ is formal if and only if $P_n A$ is formal. \end{prop}

If $A$ is a formal minimal Sullivan algebra whose cohomology is finite-dimensional and concentrated in degrees $\leq k$, the formality of $P_n A$ for $n > 2k$ can alternatively be quickly recovered from \cite[Theorem 3.1]{FM05}.

\section{A quadruple Massey product pulling back trivially}\label{counterexample}

\subsection{Construction of the example}

Our goal is to prove the following statement:

\begin{thm}\label{thm: counterex} There is a non-zero degree map of cohomologically connected rational Poincar\'e duality cdga's $P_1 \xrightarrow{f} P_2$ such that $P_1$ carries a non-trivial quadruple Massey product, which becomes trivial in $P_2$. That is, there are cohomology classes $[w],[x],[y],[z] \in H(P_1)$ such that the Massey product $\langle [w],[x],[y],[z] \rangle$ is defined and does not contain zero, and $$0 \in \langle [f(w)], [f(x)], [f(y)], [f(z)] \rangle.$$ \end{thm}

We first construct cdga's with finite-dimensional cohomology $A,B$ (not satisfying Poincar\'e duality), together with a cdga morphism $f\colon A\rightarrow B$ and a differential graded $A$-module homomorphism $r\colon B\rightarrow A$ sending $1\mapsto 1$, such that $A$ carries a non-trivial quadruple Massey product which $f$ sends to a trivial one. Then Poincar\'e dualizing for large enough $n$ (namely, $n\geq 7$) and applying \Cref{prop:PDnotFormal} (and \Cref{trivialMassey}) will yield a map $P_n(A) \to P_n(B)$ with the desired properties. The cdga's $A$ and $B$ we will construct here will be free as graded algebras, facilitating the further investigation of $B$ in \Cref{homotopicalpropertiesofB}. Another (related but more ad hoc) example proving \Cref{thm: counterex} is provided in \Cref{alternativeB}. 

\begin{rem}
In view of Taylor's result, see Proposition \ref{prop:Taylor}, the above datum $f,r\colon A\leftrightarrows B$ with $A$ carrying a non-trivial triple Massey product which becomes trivial in $B$ can not exist (in particular, there is no such $r$ in \Cref{noPD}). 

Indeed, consider a triple Massey product $\langle [x],[y],[z]\rangle$ in $A$. Then, as the triple Massey product $\langle [f(x)],[f(y)],[f(z)]\rangle$ vanishes in $B$, we can find a defining system $a,b\in B$ with $da= f(x)f(y)$, $db=f(y)f(z)$ such that $af(z)-(-1)^{|x|}f(x)b$ is exact. But then using that $r$ is a dg-$A$-module morphism we find that $r(a),r(b)$ is a defining system for $\langle [x],[y],[z]\rangle$ and the representing cocycle
\[r(a)z-(-1)^{|x|} x r(b)=r\left( af(z)-(-1)^{|x|}f(x)b\right) \] is exact.
This shows triviality of the Massey product $\langle [x],[y],[z]\rangle$. 

In order to motivate what is happening in the example we will construct for \Cref{thm: counterex}, it is rather instructive to check where the above argument fails for quadruple Massey products. To this end consider a quadruple Massey product $\langle [w],[x],[y],[z]\rangle$ in $A$. (For simplicity of notation, assume these classes are in even degrees.) As before, choose a defining system $a,b,c,g,h$ for $$\langle [f(w)],[f(x)],[f(y)],[f(z)]\rangle$$ such that $da=f(w)f(x)$, $db=f(x)f(y)$, $dc=f(y)f(z)$, $dg=af(y)-f(w)b$ and $dh=bf(z)-f(x)c$. While it still holds that $r(a),r(b),r(c),r(g),r(h)$ is a defining system for $\langle [w],[x],[y],[z]\rangle$ it is in general no longer true that the cocycles representing the Massey products get mapped to one another, i.e.\ we might have
\[ wr(h)+r(a)r(c)+zr(g)\neq r(f(w)h+ac+f(z)g)\]
if $r(ac)\neq r(a)r(c)$, which can happen since $r$ is not fully multiplicative. In particular the right hand side being exact does not force the left hand side to be so. In other words: while a non-trivial triple Massey product would obstruct the construction of the module retract $r$ in the counterexample below, the freedom of choosing $r(ac)$ will allow us to construct $r$ even in the presence of a non-trivial quadruple Massey product. 
\end{rem}

We begin with the construction of $A$. Set $(A,d):=(\Lambda (V^{\leq 5})\otimes \Lambda(V^{\geq 6}),d)$, where $$V^{\leq 5}=\langle X,Y,a,b,c,e,f,h,i\rangle$$ with\vspace{0.2cm}

\begin{center}
\begin{tabular}{ c | c | c }
degree & generators & differential \\ \hline
$2$ & $X,Y$ & $X,Y\mapsto 0$ \\
$3$ & $a,b,c$ & $a\mapsto X^2 \quad b\mapsto XY \quad c\mapsto Y^2$\\
$4$ & $e,f$ & $e\mapsto Ya-Xb\quad f\mapsto Yb-Xc$\\
$5$ & $h,i$ & $h\mapsto Xe+ab\quad i\mapsto Yf+bc$

\end{tabular}
\end{center}
and $V^{\geq 6}$ is a vector space  which we construct inductively in order to eliminate all cohomology in degrees $\geq 7$. To be precise, we first choose cycles representing a basis for degree $7$ cohomology. Then for each of these elements introduce a generator in $V^6$ and map it to the chosen cycle under the differential. The resulting algebra will have trivial degree $7$ cohomology while cohomology in degrees $\leq 6$ remains unchanged. Now repeat this process inductively for all higher degrees.

\begin{lem}\label{lem:generatorsinA}
The cohomology of $(A,d)$ is generated by the linearly independent cohomology classes of the cocycles $1,X,Y,m$, where $m=Ye+ac+Xf$. Furthermore the Massey product $\langle [X],[X],[Y],[Y]\rangle\subseteq H^6(A)$ equals $\{[m]\}$, so is non-trivial.\end{lem}

\begin{proof}
Clearly $1,X,Y$ generate cohomology in degrees $\leq 2$. Furthermore $A^3=\langle a,b,c\rangle$ maps isomorphically onto $\Lambda^2 (X,Y)$ so there is no cohomology in degree $3,4$ in $\Lambda (X,Y,a,b,c)$. This changes in degree $5$, where the $\ker d$ is generated by $Ya-Xb$, $Yb-Xc$. Note that the corresponding cohomology classes do indeed form a basis of $H^5(\Lambda(X,Y,a,b,c),d)$, since $d$ vanishes on the degree $4$ span of the above generators. Thus after introducing $e,f$ we obtain $H^5(\Lambda (X,Y,a,b,c,e,f),d)=0=H^4(\Lambda (X,Y,a,b,c,e,f),d)$. At this point we compute that the degree $6$ part of $\ker d$ is $\langle Xe+ab,m,Yf+bc\rangle\oplus \Lambda^3(X,Y)$. The differential maps the degree $5$ span of the above generators onto $\Lambda^3(X,Y)$ so the cocycles in the left hand factor yield a basis for the cohomology at this stage, after introducing $h,i,V$ only the class of $m$ remains, generating $H^6(A)$. This proves the first part of the lemma. The reader can also verify this with the ``Commutative Differential Graded Algebras'' module in \cite{Sage}: 

{\tiny \begin{align*} &T.<X,Y,a,b,c,e,f,h,i> = \textrm{GradedCommutativeAlgebra(QQ, degrees = (2,2,3,3,3,4,4,5,5))} \\
&A = T.\textrm{cdg}\_{\textrm{algebra}}(\{a: X*X, b: X*Y, c: Y*Y, e: Y*a - X*b, f: Y*b - X*c, h: X*e + a*b, i: Y*f + b*c\}) \\
&[A.\textrm{cohomology}(i) \textrm{ for } i \textrm{ in } [1..7]] \end{align*} }

When writing down a defining system for the Massey product $\langle [X],[X],[Y],[Y]\rangle$, the unique choice for the primitives of $X^2,XY,Y^2$ is $a,b,c$. When choosing primitives $p_1,p_2$ for the cocycles $Ya-Xb$ and $Yb-Xc$, we get $p_1=e+\alpha_1$, $p_2=f+\alpha_2$ for some $\alpha_i\in (\ker d)^4=\Lambda^2(X,Y)$. Then the resulting cocycle representing a class in $\langle [X],[X],[Y],[Y]\rangle$ is $m+Y\alpha_1+X\alpha_2$. Independently of the choice of the exact $\alpha_i$, this is cohomologous to $m$. 
\end{proof}

Now we come to the construction of $B$, which we will define as $(A\otimes \Lambda W,d)$. Up until degree $5$ the generators of $W$ and their images under $d$ are given as follows:
\vspace{0.2cm}

\begin{center}
\begin{tabular}{ c | c | c }
degree & generators & differential\\ \hline
$3$ & $\alpha,\gamma$ & $\alpha,\gamma\mapsto 0$ \\
$4$ & $s_{X\alpha},s_{Y\alpha},s_{X\gamma},s_{Y\gamma}$ & $s_{*}\mapsto *$\\
$5$ & $t_i$, $i=1,\ldots,9$ & $t_i\mapsto v_i$

\end{tabular}
\end{center}
\vspace{0.2cm}
where
\smallskip
\begin{center}
\begin{tabular}{l l l }
$v_1=m-\alpha\gamma,$ & $v_2=Ys_{X\alpha}-Xs_{Y\alpha},$ & $v_3=Ys_{X\gamma}-Xs_{Y\gamma},$\\
 $v_4=a\alpha-Xs_{X\alpha},$ & $v_5=b\alpha-Xs_{Y\alpha},$ & $v_6=c\alpha-Ys_{Y\alpha},$\\
 $v_7=a\gamma-Xs_{X\gamma},$ & $v_8=b\gamma-Xs_{Y\gamma},$
& $v_9=c\gamma-Ys_{Y\gamma}.$
\end{tabular}
\end{center}
\smallskip

We check that at this stage
$H^*(A\otimes \Lambda (\alpha,\gamma,s_{X\alpha},s_{Y\alpha},s_{X\gamma},s_{Y\gamma}),d)$ is trivial in degree $5$, and in degree $6$ a basis is represented by $m$ and the cocycles $v_i$, $i=1,\ldots,9$.
We finish the construction of $B$ by defining $W^{\geq 6}$ so that $H^{\geq 7}(B)=0$ by inductively killing all cohomology in degrees $\geq 7$. We will not need to describe this last step explicitly.

\begin{rem}
The cdga $(A\otimes\Lambda W^{\leq 5},d)$ is minimal and one can carry out the construction such that $(B,d)$ is minimal. In fact $A\rightarrow B$ is a relative minimal model. Cohomologically $H^*(B)=H^*(A)\oplus \langle [\alpha],[\gamma]\rangle$, where $[\alpha\gamma]$ is a nontrivial generator of $H^6(A)$.
\end{rem}

Now we verify that the quadruple product $\langle [X], [X], [Y], [Y] \rangle$ becomes trivial in $B = A \otimes \Lambda W$ under the inclusion $A \xrightarrow{j} A \otimes \Lambda W$:

\begin{lem}\label{quadruplevanishes} We have $0 \in \langle j [X], j [X], j [Y], j [Y] \rangle$. \end{lem}

\begin{proof} For simplicity we omit explicitly mentioning the inclusion map $j$. Choose $X, Y$ as representatives of $[X], [Y]$, and make the following choice of primitives: $d(a - \alpha) = X^2$, $db = XY$, $d(c+\gamma) = Y^2$. With these choices, the triple product $\langle [X], [X], [Y] \rangle$ is represented by $(a-\alpha)Y - Xb$, for which we choose the primitive $d(e-s_{Y\alpha}) = (a-\alpha)Y - Xb$. For the triple product $\langle [X], [Y], [Y] \rangle$, we have $d(f - s_{X\gamma}) = bY - X(c+\gamma)$ for the given representative. Hence the quadruple product, with these choices of primitives, is represented by \begin{align*} &(e-s_{Y\alpha})Y + (a-\alpha)(c+\gamma) + X(f-s_{X\gamma}) = d(t_1 + t_6 + t_7). \qedhere \end{align*}
\end{proof}

It remains to construct a dg-$A$-module retract of the map $j\colon A\rightarrow A\otimes \Lambda W$. In order to do this we recall the following:

\begin{defi}
Let $A$ be a dga and $(M,d)$ be a dg-$A$-module. Then a semi-free extension of $(M,d)$ is a dg-$A$-module of the form $(M\oplus (A\otimes V),d)$, where $V$ is a graded vector space and $d(1\otimes V)\subset M$.\end{defi}

For us this concept is helpful due to the following standard lemma. Part (1) is an immediate observation, while part (2) is a more explicit form of \cite[Lemma 14.1]{FHT12} which will prove useful when dealing with the explicit example.

\begin{lem}\label{lem:dgmodules}
\begin{enumerate}
    \item Let $f\colon M\rightarrow N$ be a morphism of dg-$A$-modules and $(M\oplus (A\otimes V),d)$ a semi-free extension of $M$. Let $(v_i)_{i\in I}$ be a basis of $V$, and let $(\alpha_i)_{i\in I}$ be a collection of elements in $N$ with $d\alpha_i= f(d v_i)$. Then $f$ extends to a morphism of dg-$A$-modules $M\oplus (A\otimes V) \to N$ by setting $f(v_i)=\alpha_i$.
    \item Let $A\rightarrow A\otimes \Lambda W$ be a relative minimal cdga with $A^1=W^1=0$. For $0\leq j\leq i$, set $V_{(i,j)}=(\Lambda^{i-j}W)^{2i-j}$. Then $\Lambda W=\bigoplus_{0\leq j\leq i} V_{(i,j)}$ and for any $(i,j)$ as above the inclusion
    \[A\otimes \left(\bigoplus_{(k,l)<(i,j)} V_{(k,l)}\right)\rightarrow A\otimes \left(\bigoplus_{(k,l)\leq(i,j)} V_{(k,l)}\right)\]
    is a semi-free extension, where we use the lexicographical order on tuples.
\end{enumerate}
\end{lem}

\begin{proof}
Part (1) is straightforward verification. For the proof of part $(2)$ we observe that due to $W=W^{\geq 2}$ we indeed have \[\Lambda W=\bigoplus_{0\leq 2k\leq l} (\Lambda^k W)^{l}= \bigoplus_{0\leq j\leq i} (\Lambda^{i-j}W)^{2i-j}.\] 
It remains to check that $d(V_{(i,j)})\subseteq A\otimes \left(\bigoplus_{(k,l)<(i,j)} V_{(k,l)}\right)$. To see this, we investigate the differential with respect to its bidegree $A\otimes ((\Lambda^p W)^q)$, where $p$ is the wordlength degree in $W$ and $q$ is the cohomological degree in $\Lambda W$.
If $p$ does not increase then $q$ decreases by at least $1$ due to minimality and $A^1=0$. Furthermore $p$ can decrease by at most $1$ in which case $q$ decreases by $2$ since $d(W)\cap A$ lies in degrees $\geq 3$. Consequently
\begin{align*}
    d\left((\Lambda^{i-j} W)^{2i-j}\right)\subset A\otimes \left((\Lambda^{\geq i-j+1} W)^{\leq 2i-j+1}\oplus  (\Lambda^{i-j} W)^{\leq 2i-j-1}\oplus (\Lambda^{ i-j-1} W)^{\leq 2i-j-2}\right)
    \end{align*}
which proves the claim. \end{proof}

Thus by this lemma, in order to define the retraction $r\colon A\otimes \Lambda W\rightarrow A$ it suffices to inductively specify images of a suitable basis of $\Lambda W$ and extend $A$-linearly. By the following lemma, no obstructions arise past a certain degree.
\begin{lem}\label{lem:extension} Let $A$ be the particular dga described above. Any morphism
\[r\colon A\otimes \left(\bigoplus_{(i,j)\leq (4,3)} V_{(i,j)}\right)\rightarrow A\]
of dg-$A$-modules extends to $A\otimes \Lambda W$.
\end{lem}
 \begin{proof}
 Recall that by part (1) of \Cref{lem:dgmodules} the only obstruction to extend $r$ over a new generator $v$ is that the class $[r(dv)]\in H^{*}(A)$ has to vanish. By definition, for $(i,j)>(4,3)$ the space $V_{(i,j)}$ is concentrated in cohomological degrees $\geq 6$ (since $V_{(i,i)}=0 $ for $i\neq 0$) while $H(A)$ is concentrated in degrees $\leq 6$.
 \end{proof} 

Furthermore note that for $(i,j)\leq (4,3)$, we have $V_{(i,j)}\subset \Lambda (W^{\leq 5})$ which means we have already computed all the required algebra generators. We define $r$ according to the following table, where we list all non-trivial $V_{(i,j)}$ with $(i,j)\leq (4,3)$ in their order of occurrence.

\vspace{0.2cm}

\begin{center}
\begin{tabular}{ c | c | c }
extension & generators & image under $r$ \\ \hline
$V_{(0,0)}$ & $1$ & $1\mapsto 1$\\
$V_{(2,1)}$ & $\alpha,\gamma$ & $\alpha,\gamma\mapsto 0$\\
$V_{(3,1)}$ & $s_{X\alpha},s_{Y\alpha},s_{X\gamma},s_{Y\gamma}$ & $s_{*}\mapsto 0$\\
$V_{(4,2)}$ & $\alpha\gamma$ & $\alpha\gamma\mapsto m$
\\
$V_{(4,3)}$ &  $t_1,\ldots,t_9$ & $t_i\mapsto 0$
\end{tabular}
\end{center}
\vspace{0.2cm}

One checks that indeed for any of the generators $v$ above we have $r(dv)=dr(v)$. Then by \Cref{lem:dgmodules} and \Cref{lem:extension} we obtain the desired retraction $r\colon A\otimes \Lambda W\rightarrow A$.

\vspace{0.5em}

Applying the Poincar\'e dualization construction, together with \Cref{prop:PDnotFormal} and \Cref{trivialMassey}, we obtain a non-zero degree map of Poincar\'e duality algebras with the desired properties, as long as $n \geq 7$. This completes the proof of \Cref{thm: counterex}. If $n\geq 8$, so that $P_nA$ and $P_nB$ are cohomologically simply connected, the example can moreover be realized by a map of simply connected spaces \cite{Su77} which satisfy rational Poincar\'e duality.

\begin{rem}\label{alternativeB}
Instead of introducing more generators to cohomologically truncate the cdgas in the previous construction one can consider the following more ad hoc variant of truncation in the spirit of \Cref{noPD}, avoiding some of the technicalities. The cdgas below, although not minimal and cohomologically larger than the previous construction, are rather similar in spirit and do also provide a counterexample for the proof of \Cref{thm: counterex}. Define
$A':=\Lambda(X,Y,a,b,c,e,f)$ with the following degrees and differential:
\begin{center}
\begin{tabular}{ c | c | c }
degree & generators & differential \\ \hline
$2$ & $X,Y$ & $X,Y\mapsto 0$ \\
$3$ & $a,b,c$ & $a\mapsto X^2 \quad b\mapsto XY \quad c\mapsto Y^2$\\
$4$ & $e,f$ & $e\mapsto Ya-Xb\quad f\mapsto Yb-Xc$.
\end{tabular}
\end{center}
 Denote by $m:=Xf+ac+Ye$ a representative for the quadruple Massey product $$\langle [X],[X],[Y],[Y] \rangle$$ which one checks to be nontrivial just as in the original construction. Set $A:=A'/(A')^{\geq 7}$. Next, define $B'=A\otimes\Lambda W$ with $
W = \mathrm{span}\langle \alpha,\gamma,s_{Y\alpha},s_{X\gamma},\Omega\rangle$
with the following degrees and differential:
\begin{center}
\begin{tabular}{ c | c | c }
degree & generators & differential \\ \hline
$3$ & $\alpha,\gamma$ & $\alpha,\gamma\mapsto 0$ \\
$4$ & $s_{Y\alpha},s_{X\gamma}$ & $s_{Y\alpha} \mapsto Y\alpha\quad s_{X\gamma} \mapsto X\gamma$\\
$5$ & $\Omega$ & $\Omega\mapsto m-Xs_{X\gamma}-Ys_{Y\alpha}+a\gamma-\alpha c-\alpha\gamma$
\end{tabular}
\end{center} 
Now, set $B:= B'/I$ where $I$ is the ideal generated by $(\Lambda W)^{\geq 7}$. Then, as an $A$-module, $B$ is free of rank $6$:
\[
B=A\oplus A\alpha\oplus A\gamma\oplus A s_{Y\alpha} \oplus A s_{X\gamma} \oplus A\Omega\oplus A\alpha\gamma .
\]
In fact it is a sequence of semi-free extensions. Thus we can define a module retract $r:B\to A$ by sending all but the first and last summand to zero, sending $1\mapsto 1$ and $\alpha\gamma\mapsto m$. This map of modules is compatible with the differential, i.e. $rd=dr$ which, according to \Cref{lem:dgmodules} can be checked directly on the generators.

The Massey product $\langle [X], [X], [Y], [Y] \rangle$ vanishes on $B$. Indeed, choose primitives $d(a-\alpha) = X^2, db = XY, d(c+\gamma) = Y^2$. Then the triple product $\langle [X], [X], [Y] \rangle$ is represented by $aY - Xb - \alpha Y = d(e-s_{Y\alpha})$, and $\langle [X], [Y], [Y] \rangle$ is represented by $bY - Xc - X\gamma = d(f-s_{X\gamma})$. With these choices of primitives, the quadruple product is represented by $$X(f-s_{X\gamma}) + (a-\alpha)(c+\gamma) + (e-s_{Y\alpha})Y = d\Omega.$$
Poincar\'e dualizing the pair $(A,B)$ with respect to the inclusion and $r$ gives another example as in \Cref{thm: counterex}.
\end{rem}

\subsection{Homotopical properties of $B$ and its Poincar\'e dualization}\label{homotopicalpropertiesofB}

We now consider Massey products and the non-formality of the cdga $B$ constructed for \Cref{thm: counterex} (we will not consider now the one constructed in \Cref{alternativeB}). Namely, we will show all Massey products on $B$ vanish, but that $B$ is not formal. We will explicitly see that one cannot make uniform choices making all Massey products on $B$ vanish.

Recall, $B$ only has cohomology in degrees 2, 3, 6, spanned by $\{[X],[Y]\}$, $\{[\alpha], [\gamma] \}$, $\{[m] = [Ye+ac+Xf]=[\alpha\gamma]\}$ respectively.

\begin{prop}\label{tripleMPonB} All triple Massey products on $B$ vanish. \end{prop}

\begin{proof} For degree reasons, we only have to consider triple Massey products involving two degree two classes and one degree three class. There are two cases: when the degree three class is the first or third entry (we only need consider one of these two possibilities), and when the degree three class is the middle entry.

In the first case, i.e. a Massey product of the form $$\langle c_1[X] + c_2[Y], c_3[X] + c_4[Y], c_5[\alpha] + c_6[\gamma] \rangle,$$ first note that $X$, $Y$, $\alpha$, $\gamma$ are the unique representatives of the appropriate cohomology classes (recall that in general the Massey product does not depend on this choice of representatives of the starting classes). Now, we have \begin{align*} (c_1 X + c_2 Y)(c_3 X + c_4 Y) &= d( (c_1c_3)a + (c_1c_4 + c_2c_3)b + (c_2c_4)c), \\ (c_3 X + c_4 Y)(c_5 \alpha + c_6 \gamma) &= d((c_3c_5)s_{X\alpha} + (c_3c_6)s_{X\gamma} + (c_4c_5)s_{Y\alpha} + (c_4c_6)s_{Y\gamma}). \end{align*} With these choices of primitives, the Massey product is represented by \begin{align*} ( (c_1c_3)a + (c_1c_4 + c_2c_3)b &+ (c_2c_4)c)(c_5\alpha + c_6\gamma) \\ &- (c_1 X + c_2 Y)((c_3c_5)s_{X\alpha} + (c_3c_6)s_{X\gamma} + (c_4c_5)s_{Y\alpha} + (c_4c_6)s_{Y\gamma}) \end{align*} \begin{align*} = c_1c_3c_5dt_4 &+ c_1c_3c_6dt_7 + c_2c_4c_5 dt_6 + c_2c_4c_6 dt_9 + c_1c_4c_5 dt_5 \\ &+ c_1c_4c_6 dt_8 + c_2c_3c_5d(t_5-t_2) + c_2c_3c_6d(t_8-t_3), \end{align*} which is exact. For the second case, i.e. a Massey product of the form $$\langle c_1[X] + c_2[Y], c_3[\alpha] + c_4[\gamma], c_5[X] + c_6[Y] \rangle,$$ we choose again $X, Y, \alpha, \gamma$ as the obvious representatives, and \begin{align*} (c_1 X + c_2 Y)(c_3 \alpha + c_4 \gamma) &= d((c_1c_3)s_{X\alpha} + (c_1c_4)s_{X\gamma} + (c_2c_3)s_{Y\alpha} + (c_2c_4)s_{Y\gamma}), \\ (c_3\alpha + c_4 \gamma)(c_5 X + c_6 Y) &= d( (c_3c_5) s_{X\alpha} + (c_3c_6) s_{Y\alpha} + (c_4c_5)s_{X\gamma} + (c_4c_6) s_{Y\gamma}).\end{align*} Then the representative of the Massey product corresponding to these choices of primitives is the image under $d$ of \begin{align*} &c_1c_3c_6t_2 + c_1c_4c_6 t_3 - c_2c_3c_5 t_2 - c_2c_4c_5t_3. \qedhere \end{align*}
\end{proof}

We saw in \Cref{quadruplevanishes} that the quadruple Massey product $\langle [X], [X], [Y], [Y] \rangle$ vanishes in $B$. We extend this to the following:

\begin{prop} All quadruple Massey products on $B$ vanish. \end{prop}

\begin{proof} For degree reasons, we only need consider quadruple Massey products of the form $$\langle c_1[X] + c_2[Y], c_3[X] + c_4[Y], c_5[X] + c_6[Y], c_7[X] + c_8[Y] \rangle.$$ As before, there are unique choices of representatives $X, Y$ for $[X], [Y]$ respectively. We will make the following choices of primitives throughout: $$X^2 = d(a-\alpha), \ \ XY = db, \ \ Y^2 = d(c+\gamma).$$ So, we have \begin{align*} (c_1 X + c_2 Y)(c_3 X + c_4 Y) &= d(c_1c_3(a-\alpha) + (c_1c_4+c_2c_3)b + c_2c_4(c+\gamma)), \\ (c_3 X + c_4 Y)(c_5 X + c_6 Y) &= d(c_3c_5(a-\alpha) + (c_3c_6+c_4c_5)b + c_4c_6(c+\gamma)), \\ (c_5 X + c_6 Y)(c_7 X + c_8 Y) &= d(c_5c_7(a-\alpha) + (c_5c_8+c_6c_7)b + c_6c_8(c+\gamma)). \end{align*} 

With obvious choices of primitives, the triple Massey product $$\langle c_1[X] + c_2[Y], c_3[X] + c_4[Y], c_5[X] + c_6[Y] \rangle$$ is then represented by \begin{align*} d\left( (c_2c_4c_5-c_1c_4c_6)s_{X\gamma} + (c_2c_3c_5 - c_1c_3c_6)s_{Y\alpha} + (c_1c_3c_6 - c_2c_3c_5)e + (c_1c_4c_6-c_2c_4c_5)f\right), \end{align*} and $\langle c_3[X] + c_4[Y], c_5[X] + c_6[Y], c_7[X] + c_8[Y] \rangle$ is represented by \begin{align*} d\left( (c_4c_6c_7 - c_3c_6c_8) s_{X\gamma} + (c_4c_5c_7 - c_3c_5c_8)s_{Y\alpha} + (c_3c_5c_8 - c_4c_5c_7)e + (c_3c_6c_8 - c_4c_6c_7)f\right).\end{align*} Then the quadruple Massey product is represented by \begin{align*} 
&d((c_2c_4c_5c_7 - c_1c_3c_6c_8)t_7 - (c_2c_3c_5c_7 - c_1c_3c_6c_7 + c_1c_4c_5c_7 - c_1c_3c_5c_8)t_5 \\ &+ (c_1c_3c_6c_7-c_2c_3c_5c_7+c_1c_3c_5c_8-c_1c_4c_5c_7)h + (c_1c_3c_6c_8-c_2c_4c_5c_7)t_1 \\ &+ (c_2c_4c_5c_8 - c_1c_4c_6c_8 + c_2c_4c_6c_7 - c_2c_3c_6c_8)(t_8 - t_3) - (c_2c_4c_5c_7-c_1c_3c_6c_8)t_6 \\ &+ (c_1c_4c_6c_8-c_2c_4c_5c_8+c_2c_3c_6c_8-c_2c_4c_6c_7)i).\qedhere 
\end{align*}
\end{proof}

\begin{cor} All Massey products on $B$ vanish. \end{cor}

\begin{proof} By the above, all triple and quadruple products vanish. For degree reasons we see that there cannot be non-trivial quintuple or higher products.
\end{proof}

Note that the above calculations show that among triple Massey products on $B$, uniform choices of primitives can be made so that they all vanish, and likewise separately for quadruple Massey products. However, we will now see that one cannot make uniform choices \emph{simultaneously} for both the triple and quadruple Massey products. More precisely we have the following:

\begin{defi} We say that all Massey products on a cdga $A$ vanish uniformly if
    there is a map $d^{-1}:\operatorname{im} d\to A$ such that $d\circ d^{-1}=\operatorname{Id}$ and for a Massey product $m=\langle a_{0,1},...,a_{r-1,r}\rangle$ one can inductively build a defining system yielding the trivial class $[0]\in m$ by setting $a_{i,j}:=d^{-1}\sum_{i<l<j}\bar{a}_{i,l}a_{l,j}$, where $\bar{a}=(-1)^{|a|+1} a$ for homogeneous elements.
\end{defi} 

\begin{lem}
    Let $A=(\Lambda V,d)$ be minimal. If $A$ is formal, then all Massey products on $A$ vanish uniformly.
\end{lem}

\begin{proof}
    There is a quasi-isomorphism $A\rightarrow H^*(A)$ whose kernel we denote by $I$. This is an acyclic differential ideal. For $\alpha\in \operatorname{im}d$ we note that $\alpha\in I$ is closed and due to acyclicity we find $\beta\in I$ with $d\beta=\alpha$. Hence we find $C\subset I$ such that $d\colon C\rightarrow\operatorname{im}d$ is an isomorphism. We define $d^{-1}$ as the inverse of this isomorphism. In that case any representative for a Massey product built using $d^{-1}$ as above will lie in the ideal $I$ as well. As any closed element in $I$ is exact this proves the lemma.
\end{proof}

\vspace{0.4em}

Now we return to our cdga $B$. Consider the element $X^2$. A primitive for this must be of the form $a + k_1\alpha + k_2 \gamma$ for some coefficients $k_1,k_2$. Consider now the triple Massey product $\langle [X], [X], [\alpha] \rangle$. There are unique choices of representatives $X, \alpha$. A primitive for $X\alpha$ must be of the form $s_{X\alpha} + c_1X^2 + c_2XY + c_3Y^2$, and the Massey product is represented by \begin{align*} (a + k_2 \gamma)\alpha &- X(s_{X\alpha} + c_1X^2 + c_2XY + c_3Y^2) \\ &= d(t_4 - X(c_1a + c_2b + c_3c)) + k_2 \gamma \alpha. \end{align*} From here we see we must choose $k_2 = 0$ in order to make $\langle [X], [X], [\alpha]\rangle$ trivial.

Similarly, we must choose $k_1 = 0$ to make $\langle [X], [X], [\gamma] \rangle$ trivial. But now, the quadruple product $\langle [X], [X], [Y], [Y] \rangle$ can not be made trivial with this choice of primitive $a$ for $X^2$. Indeed, taking the unique representatives $X,Y$, general primitives are given by \begin{align*} XY &= d(b + c_1 \alpha + c_2 \gamma), \\ Y^2 &= d(c+c_3\alpha + c_4 \gamma).\end{align*} Then the left triple product is represented by $aY - Xb - c_1 X\alpha - c_2 X\gamma$, with $e - c_1 s_{X\alpha} - c_2 s_{X \gamma} + d\phi$ being a general choice of primitive (where $\phi$ is an arbitrary element of degree three). The right triple product is represented by $bY + c_1Y\alpha  + c_2 Y\gamma - Xc - c_3 X\alpha - c_4 X \gamma$, with general choice of primitive given by $f + c_1 s_{Y\alpha} + c_2 s_{Y\gamma} - c_3 s_{X\alpha} - c_4 s_{X\gamma} + d\phi'$ for some $\phi'$.

With these choices of primitives, the quadruple product is represented by \begin{align*} Xf + ac + Ye &+ c_1 Xs_{Y\alpha} + c_2 Xs_{Y\gamma} - c_3 Xs_{X\alpha} - c_4 Xs_{X\gamma} \\ &+ c_3 a\alpha + c_4 a\gamma - c_1 Y s_{X\alpha} - c_2 Ys_{X\gamma} + d(X \phi' + \phi Y).\end{align*}

Note that $m = Xf + ac + Ye$ appears, while $\alpha \gamma$ does not; hence this element represents a non-zero cohomology class. We conclude that the Massey products on $B$ do not vanish \emph{uniformly}. We therefore have:

\begin{prop}\label{Bnotformal} $B$ is not formal. \end{prop}

Note that this also follows from \cite[Theorem B]{MSZ23}. We now give a third proof by considering the Poincar\'e dualization $P_n(B)$. Fix a basis for $D_n(B)$, namely the dual basis to that given by the monomials in the chosen generators for $B$. If $p$ is such a monomial, $\widehat{p}$ will denote the corresponding element in the dual.

\begin{lem}\label{nontrivialquadrupleMP} The quadruple Massey product $\langle [X], [Y], [Y], [\widehat{Ye} - \widehat{\alpha \gamma}] \rangle$ is non-trivial in $P_n(B)$ for $n\geq 12$. \end{lem}

\begin{proof}
For simplicity of signs, we assume $n$ is even, though the conclusion does not require this assumption. 
Choose $\widehat{Ye} - \widehat{\alpha \gamma}$ as a representative for its cohomology class; general primitives for the adjacent pairwise products are given by \begin{align*} XY &= d(b + k_1 \alpha + k_2 \gamma + d\varphi_1), \\ Y^2 &= d(c + k_3 \alpha + k_4 \gamma + d\varphi_2), \\ Y(\widehat{Ye} - \widehat{\alpha \gamma}) &= \widehat{e} = d(k_5 \widehat{Ya} - k_6\widehat{Xb} + d\varphi_3), \end{align*} where the $k_i$ are coefficients such that $k_5+k_6 = 1$, and $\varphi_1,\varphi_2\in (P_n B)^2$ and $\varphi_3\in (P_nB)^{n-6}$ are arbitrary. The $d\varphi_i$ terms can be chosen exact since $H^3(P_nB)=H^3(B)=\langle \alpha,\gamma\rangle$ and $H^{n-5}(P_nB)=0$.

The left triple product $\langle [X], [Y], [Y] \rangle$ is thus represented by $$(bY-Xc) + k_1Y\alpha + k_2Y\gamma - k_3X\alpha - k_4X\gamma + d(Y\varphi_1- X\varphi_2), $$ with general choice of primitive $$f + k_1s_{Y\alpha} + k_2s_{Y\gamma} - k_3s_{X\alpha} - k_4s_{X\gamma} + Y\varphi_1 - X\varphi_2 + d\eta,$$ where $\eta$ is a degree 3 element (where we use that $H^4(P_nB)=0$).

The right triple product $\langle [Y], [Y], [\widehat{Ye} - \widehat{\alpha \gamma}]\rangle$ is represented by $$(c+k_3\alpha + k_4\gamma + d\varphi_2)(\widehat{Ye} - \widehat{\alpha \gamma}) - Y(k_5\widehat{Ya} - k_6 \widehat{Xb} + d\varphi_3) = -k_3\widehat{\gamma} - k_4\widehat{\alpha} - k_5\widehat{a} +d(\varphi_2(\widehat{Ye}-\widehat{\alpha\gamma})-Y\varphi_3).$$

Since $H^{n-3}(P_nA)=\langle \widehat{\alpha},\widehat{\gamma}\rangle$, we see that necessarily $k_3 = k_4 = 0$ if the above expression is to be exact. A general primitive for the above representative of the right triple product is then given by $k_5\widehat{X^2} + d\varepsilon +\varphi_2(\widehat{Ye}-\widehat{\alpha\gamma})-Y\varphi_3$, where $\varepsilon$ is an element of degree $n-5$ (using that $H^{n-4}(P_nB)=0$). Thus, any representative of the quadruple product is of the form 
\begin{align*} (f+k_1s_{Y\alpha} &+ k_2s_{Y\gamma} +(Y\varphi_1-X\varphi_2) + d\eta)(\widehat{Ye} - \widehat{\alpha\gamma}) \\ &+ (b+k_1\alpha + k_2\gamma + d\varphi_1)(k_5\widehat{Ya} - k_6\widehat{Xb} + d\varphi_3) + X(k_5\widehat{X^2} - Y\varphi_3 + \varphi_2(\widehat{Ye}-\widehat{\alpha\gamma})+ d\varepsilon).\end{align*}

Note that $d\eta$ is the sum of an element in $D_n(B)$ and an element in the span of $X^2,Y^2,XY$ and $\varphi_1, \varphi_2$ are sums of elements in $D_n(B)$ and elements in the span of $X,Y$. Thus, since $X^2, XY, Y^2$ are exact, also $(d\eta)(\widehat{Ye} - \widehat{\alpha \gamma})$,  $Y\varphi_1(\widehat{Ye}-\widehat{\alpha\gamma})$,  $X\varphi_2(\widehat{Ye}-\widehat{\alpha\gamma})$, and $d\varphi_1(k_5\widehat{Ya}-k_6\widehat{Xb}+d\varphi_3)=d(\varphi_1d\varphi_3)$ are exact.

Hence, up to these exact terms, a general representative for the quadruple product is $$k_6 \widehat{X} + bd\varphi_3 + d((k_1\alpha + k_2\gamma)\varphi_3) + k_5 \widehat{X} - XY\varphi_3 + d(X\varepsilon),$$ which, since $k_5+k_6 = 1$, is cohomologous to $\widehat{X}$. \end{proof}

\begin{proof}[Proof of \Cref{Bnotformal}] This follows from the later \Cref{AformaliffPAformal}, which tells us that formality of $A$ would imply formality of $P_n A$, which contradicts the existence of a non-trivial quadruple Massey product noted above. \end{proof}

\section{Poincar\'e dualization for $A_{\infty}$-algebras}\label{inftyPDsection}

We now discuss how the Poincar\'e dualization extends to the category of $A_{\infty}$-algebras; we point the reader to \cite{keller}, \cite{LH}, \cite{lv} for detailed treatments of $A_\infty$ and $C_\infty$-algebras, along with $A_\infty$-bimodules. We will use the notation and conventions in \cite[Section 2]{MSZ23}.

\begin{rem}
   There are a priori different notions of quasi-isomorphisms when considering unital or non-unital algebras and morphisms and $A_\infty$ and $C_\infty$-algebras. However, two strictly unital $A_\infty$-algebras are quasi-isomorphic through strictly unital $A_\infty$-morphisms iff they are quasi-isomorphic through arbitrary $A_\infty$-morphisms. Furthermore, (resp. unital) dga's are weakly equivalent iff they are weakly equivalent as (resp. unital) $A_\infty$-algebras. The same statements holds for (resp. unital) cdga's and $C_\infty$-algebras. Furthermore, two (unital) cdga's (resp. $C_\infty$-algebras) are quasi-isomorphic as dga's iff they are quasi-isomorphic as (unital) dga's (resp. $A_\infty$-algebras). See \cite{HM12}, \cite{LH}, \cite{lv} and the discussion in \cite{CPRNW19}.
\end{rem}

Let $(A,m_*)$ be an $A_\infty$-algebra with finite-dimensional cohomology. Fix some $n$ and define $D_n A$ and $P_n A$ as in \Cref{poincaredualization}. We extend the definition of the $m_i$ to all of $P_n A$. Set $m_i(\bullet)=0$ whenever two or more inputs come from $D_n A$. For $\varphi\in (D_n A)^{n-k}=(A^k)^\vee$ we extend the definition of $m_i$ by setting \begin{equation}\label{formulaformi} m_i(x_1,\ldots,x_{l-1},\varphi,x_{l+1},\ldots,x_i)(b) = \varepsilon \cdot \varphi(m_i(x_{l+1},\ldots,x_i,b,x_1,\ldots, x_{l-1})) \end{equation}
for $x_j,b\in A$, where $\varepsilon = (-1)^{li+\Sigma_1^{l-1}\cdot(\Sigma_{l+1}^i+|b|+|\varphi|)+i|\varphi|}$ with $\Sigma_s^t:= \sum_{j=s}^t|x_j|$. 

This is precisely the formula that makes $D_n A$ an $A_\infty$-$A$-bimodule as in \cite[Lemma 2.9]{Tr08} (see the corrected sign in \cite{Tr11}). We then immediately have the following:

\begin{prop}\label{prop:Ainfty-dualization}
This defines an $A_\infty$-algebra structure on $P_n A$. \end{prop}

\begin{lem}\label{mapfromPD} Let $A \xrightarrow{f} B$ be a map of $A_\infty$-algebras, and $D_n A \xrightarrow{\phi} B$ a map of $A_\infty$-$A$-bimodules, where $B$ obtains its bimodule structure from $f$ (see e.g. \cite[\S 2.4]{MSZ23}). Then $f$ extends to a map of $A_\infty$-algebras $P_n A \xrightarrow{F} B$. 
\end{lem}

\begin{proof} On inputs all from $A$, we set $F_i = f_i$. If two or more inputs lie in $D_n A$ we set $F_i = 0$. For one input $\varphi$ in $D_n A$, we set $$F_i(a_1, \ldots, a_{j-1}, \varphi, a_{j+1}, \ldots a_i) = \phi_i(a_1, \ldots, a_{j-1}, \varphi, a_{j+1}, \ldots a_i).$$ That $F_i$ is a morphism of $A_\infty$-algebras follows from the corresponding equations for $f$ and the $A_\infty$-$A$-bimodule map $\phi$. \end{proof}

Since an $A_\infty$-$A$-bimodule morphism $B \to A$ induces a morphism of $A_\infty$-$A$-bimodules $D_n A \to D_n B$ (see e.g. \cite[p.9]{C08}), we obtain in particular the following generalization of \Cref{lem:naturality}:

\begin{cor}\label{extendinfty} Given an $A_\infty$-algebra morphism $f\colon A\rightarrow B$ and a morphism $r\colon B\rightarrow A$ of $A_\infty$-$A$-bimodules, we obtain an $A_\infty$-algebra morphism $P_n A \to P_n B$ extending $f$. \end{cor}

\begin{proof} Apply \Cref{mapfromPD} with $B\oplus D_n B$ being the target $A_\infty$-algebra. \end{proof}

A nice feature of $A_\infty$-bimodules is the fact that the passage to the localization at quasi-isomorphisms can be achieved by passing to the quotient category with respect to a certain homotopy relation among morphisms (see \cite[p. 94]{LH}). In particular, quasi-isomorphisms are invertible up to homotopy. While we will not need to make the homotopy relation explicit, we will make use of the following consequence:

\begin{lem}\label{lem: quasi inverse}
    Let $A$ be an $A_\infty$-algebra. Every quasi-isomorphism $f\colon M\rightarrow N$ of $A_\infty$-$A$-bimodules admits a quasi-inverse, i.e.\ a morphism $N\rightarrow M$ which is inverse to $f$ on the level of cohomology.
\end{lem}

\begin{cor}\label{cor:qi-invariance}
If $A$ and $B$ are weakly equivalent (i.e. connected by a zigzag of quasi-isomorphisms), then $P_n A$ and $P_n B$ are weakly equivalent.
\end{cor}

\begin{proof} Given a quasi-isomorphism $A \to B$, we can find a quasi-inverse $B \to A$ of $A_\infty$-$A$-bimodules, by \Cref{lem: quasi inverse}. Then as in \Cref{extendinfty} we obtain a morphism $P_n A \to P_n B$ of $A_\infty$-algebras, which is a quasi-isomorphism by construction.
\end{proof}

\begin{cor} If $A$ and $B$ are weakly equivalent dga's, then $P_n A$ and $P_n B$ are weakly equivalent (also as dga's). If $A$ and $B$ are furthermore cdga's, then $P_n A$ and $P_n B$ are weakly equivalent as cdga's by \cite[3.2--3.4]{CPRNW19}. \end{cor}

From here we immediately obtain:

\begin{cor}\label{AformaliffPAformal} If the cdga (or dga) $A$ is formal, then $P_nA$ is formal. Combining this with \Cref{formaldualization}, we have that $A$ is formal if and only if $P_n A$ is formal. \end{cor}

We do not know how to prove this using cdga's only, as the Poincar\'e dualization is functorial only for maps with a bimodule retract. Not every quasi-isomorphism has a quasi-inverse in the category of dg-bimodules but it does in $A_\infty$-bimodules.

\section{A non-formal non-zero degree map}\label{nonformalmapsection}

Let $Ho=Ho(cdga)$ be the homotopy category of cdga's, i.e. the localization at all quasi-isomorphisms. Recall that a cdga $A$ is formal if and only if it is isomorphic to $H(A)$ in $Ho$. There is a natural generalization of this condition to maps as follows: Let \[
I:=\{\bullet_1\longrightarrow \bullet_2\}\]
be a category with two objects and one morphism between them. Then let $Ho^I:=Fun(I,Ho)$ be the functor category, i.e. objects are maps $[A\to B]$ in the homotopy category and morphisms are commutative squares in $Ho$.
For every map $f:A\to B$, one has two natural objects of $Ho^I$: 
\[[f]:=[f:A\to B]\text{ and }[H(f)]:=[H(f):H(A)\to H(B)].\] In analogy with the case of objects, we say:
\begin{defi}
   A map $f$ of cdga's is called formal if $[f]\cong [H(f)]$ in $Ho^I$.
\end{defi}
For example, a cdga is formal if and only if the identity map is formal. Note that formality of $f$ implies formality of source and target. We argue that this notion coincides with several existing notions of formality of maps (which are known to be equivalent) in the literature, and give a characterization in terms of $C_\infty$-algebras that we will use later on.

Recall that for any cofibrant object $C$ in the category of cdga's (in the projective model structure, e.g.\ the cdga underlying a minimal model), morphisms $[C,A]$ in $Ho$ can be identified with homotopy classes of maps, and there is a ``lifting property'' saying that for any quasi-isomorphism $A\to B$, composition induces an isomorphism $[C,A]\cong [C,B]$. In particular, for any choice of minimal models $M_A\to A$, $M_B\to B$ and map $f:A\to B$, there exists a unique homotopy class of maps $M_A\xrightarrow{M_f} M_B$ making the following diagram commute up to homotopy:

\begin{equation}\begin{tikzcd}\label{induced map}
A\ar[d,"f"]&M_A\ar[l]\ar[d, "M_f"]\\
B&M_B\ar[l]
\end{tikzcd}
\end{equation}
\begin{lem}\label{formalmap}
    Let $f:A\to B$ be a map of connected cdga's. The following conditions are equivalent:
    \begin{enumerate}[(a)]
        \item\label{itf: formal} The map $f$ is formal.
        \item\label{itf: long diagram} There is a diagram 
        $$\begin{tikzcd}
A \arrow[d, "f"'] & A_1 \arrow[l, "\sim"'] \arrow[r, "\sim"] \arrow[d] & \cdots & A_r \arrow[l, "\sim"'] \arrow[d] \arrow[r, "\sim"] & H(A) \arrow[d, "H(f)"] \\
B                 & B_1 \arrow[l, "\sim"'] \arrow[r, "\sim"]           & \cdots & B_r \arrow[l, "\sim"'] \arrow[r, "\sim"]           & H(B)                \end{tikzcd}$$
in the category of cdga's, which commutes up to homotopy\footnote{We can take left homotopy or right homotopy. For non-cofibrant sources of morphisms, these notions differ (and are generally not even equivalence relations), but by the equivalence with \ref{itf: diagram}, either choice works.},where the horizontal morphisms are quasi-isomorphisms \cite[(2.4)]{S23}.
        \item\label{itf: diagram}
Let $M_A \xrightarrow{\sim} A$ and $M_B \xrightarrow{\sim} B$ be minimal models. Then there is a diagram 
$$\begin{tikzcd}
A \arrow[d, "f"'] & M_A \arrow[l, "\sim"'] \arrow[r, "\sim"] \arrow[d] & H(A) \arrow[d, "H(f)"] \\
B                 & M_B \arrow[l, "\sim"'] \arrow[r, "\sim"]           & H(B)                  
\end{tikzcd}$$
in the category of cdga's, which commutes up to homotopy \cite[(4.2)]{S23}.
        \item\label{itf: short diagram} 

Let $M_A \xrightarrow{\sim} A$ and $M_B \xrightarrow{\sim} B$ be minimal models. Then there is a diagram 
$$\begin{tikzcd}
M_A \arrow[r, "\sim"] \arrow[d, "M_f"'] & H(A) \arrow[d, "H(f)"] \\
M_B \arrow[r, "\sim"]           & H(B)                  
\end{tikzcd}$$
in the category of cdga's, which commutes up to homotopy \cite[p. 260]{DGMS75}.
        \item\label{itf: Cinfty}
        Consider $H(A)$ and $H(B)$ as $C_\infty$-algebras with trivial higher operations. Then there are strictly unital $C_\infty$-algebra quasi-isomorphisms $M_A \rightarrow H(A)$ and $M_B \rightarrow H(B)$ such that the diagram $$\begin{tikzcd}
M_A \arrow[d, "M_f"'] \arrow[r, "\sim"]& H(A)  \arrow[d, "H(f)"]  \\
M_B  \arrow[r, "\sim"]                                   & H(B) 
\end{tikzcd}$$ in the category of strictly unital $C_\infty$-algebras commutes up to homotopy. Here $H(f)$ is the induced morphism on cohomology, with trivial higher components.
    
    \end{enumerate}
\end{lem}

\begin{proof}

Clearly, \ref{itf: diagram} $\Rightarrow$ \ref{itf: short diagram}, \ref{itf: long diagram}, \ref{itf: formal} and by definition of $M_f$, \ref{itf: short diagram}$\Rightarrow$\ref{itf: diagram}. Using the lifting property, one sees \ref{itf: long diagram}$\Rightarrow$\ref{itf: diagram}. For \ref{itf: formal}$\Rightarrow$\ref{itf: diagram}, we note that by \eqref{induced map} one is reduced to finding maps $M_A\to H(A)$ and $M_B\to H(B)$ making the right diagram commute up to homotopy. By assumption, we find these in the homotopy category (i.e. zigzags of maps). Now one applies the lifting property repeatedly.

The implication \ref{itf: short diagram}$\Rightarrow$\ref{itf: Cinfty} is immediate. Conversely, assuming \ref{itf: Cinfty}, the horizontal arrows in the homotopy category of $C_\infty$-algebras are represented by cdga morphisms and \ref{itf: short diagram} follows. To see this using only the homotopy theory of non-unital $C_\infty$-algebras, one may use the fact that all objects in the diagram are canonically augmented and arrows respect these augmentations due to cohomological connectedness and strict unitality by arguing as follows: we pass to the non-unital category by projecting onto positive degrees, invoke the fact that the inclusion of non-unital cdga's into $C_\infty$-algebras induces an equivalence on the homotopy categories, find a cdga representative of the arrows by using that $M_A^+$, $M_B^+$ are cofibrant as non-unital cdga's and pass back to the unital category by extending the maps unitally to degree $0$.
\end{proof}

\begin{rem} Analogously, one defines formality for any diagram in the category of cdga's. Note that there is potential ambiguity in what one might mean by formality of a quasi-isomorphism $A\rightarrow A$ of a cdga $A$. If we consider the automorphism as an object in the functor category $Ho^I$, then it is easy to see that it is formal if and only if the cdga itself is formal. However, one can also consider the automorphism as an object of the functor category $Ho^{C_2}$, where $$C_2 := \{\bullet \rcirclearrowleft\}$$ is the category with one object and one non-trivial morphism, in which case formality of the morphism is a non-trivial condition beyond just formality of the cdga. 

\end{rem}

Now we construct a dominant non-formal map between formal Poincar\'e duality cdgas. Consider the formal cdgas $A=(H(S^2), d = 0)$ and $B=(H(S^2\vee (S^3\times S^4)), d= 0)$, which as vector spaces are given by $A =\langle 1,\alpha\rangle_\mathbb{Q}$ and $B=H(S^2\vee (S^3\times S^4))=\langle 1, \alpha, \beta,\gamma, \beta\gamma\rangle_\mathbb{Q}$ with $\alpha,\beta,\gamma$ of degree $2,3,4$ respectively. Between these we have the strictly unital $C_\infty$-morphism $f\colon A\rightarrow B$ with $f_1(\alpha)=\alpha$, $f_2(\alpha\otimes\alpha)=\beta$ and $f_k=0$ for $k\geq 3$.

Now in order to pass to dualizations we define the retract map $r\colon B\rightarrow A$ as the $A_\infty$-bimodule morphism $r$ with $r_1(\alpha)=\alpha$, $r_1(\beta) = r_1(\gamma) = r_1(\beta \gamma) = 0$, and vanishing higher components. Thus for $n\in \mathbb{N}$ we obtain an extension $f\colon P_n A\rightarrow P_n B$ by dualizing $r$ as in \Cref{extendinfty}. Using that $r$ has no higher components one quickly checks that the extension is again a strictly unital $C_\infty$-morphism.

\begin{prop}\label{prop:automorphism2ary}
    If $n\geq 8$, there do not exist $A_\infty$-automorphisms $\varphi$ of $P_n A$ and $\psi$ of $P_n B$ such that $(\psi \circ f\circ \varphi)_2$ vanishes on $\alpha \otimes \alpha$.
\end{prop}

\begin{proof}
     First note that any automorphism $\varphi$ of $P_n A$ satisfies $\varphi_2(\alpha\otimes \alpha)=0$  and that $\varphi_1(\alpha)$ is a non-zero multiple of $\alpha$. Hence it suffices to show that $(\psi\circ f)_2(\alpha \otimes \alpha) \neq 0$ for any automorphism $\psi$ of $P_n B$. Note that $\psi_1(\beta) = c \beta$ for some non-zero rational number $c$. Assuming on the contrary that there is an automorphism $\psi$ of $P_n B$ such that $(\psi \circ f)_2(\alpha \otimes \alpha) = 0$, we have
    \[0=(\psi\circ f)_2(\alpha\otimes\alpha)=\psi_1(f_2(\alpha\otimes \alpha))\pm \psi_2(f_1(\alpha)\otimes f_1(\alpha))=c\beta\pm\psi_2(\alpha \otimes\alpha).\] Since $m_2(\alpha\otimes \alpha)=0=m_2(\alpha\otimes \gamma)$, the morphism equation for $\psi$ applied to $\alpha\otimes\alpha\otimes \gamma$ yields
    \[0=m_2(\psi_2(\alpha\otimes\alpha)\otimes\psi_1(\gamma)) \pm m_2(\psi_1(\alpha)\otimes\psi_2(\alpha\otimes\gamma))=\pm m_2(c \beta \otimes \psi_1(\gamma)) \pm m_2(\psi_1(\alpha)\otimes\psi_2(\alpha\otimes\gamma)).\] Note that $\psi_1(\alpha)$ is a non-zero scalar multiple of $\alpha$, and $\psi_1(\gamma)$ is a non-zero scalar multiple of $\gamma$. Further note that multiplication with $\alpha$ is trivial on $B$ in positive degrees, and hence multiplication with $\alpha$ is trivial on $P_n B$ in degrees below $n-2$. As $\psi_2(\alpha\otimes\gamma)$ has degree $5$, it follows that $m_2(\alpha\otimes\psi_2(\alpha\otimes\gamma))=0$, leaving us with the contradiction $0=m_2(\beta\otimes \gamma)$.
\end{proof}

\begin{lem}\label{lem:homotopy2ary}
    If $g,g'\colon P_n A\rightarrow P_n B$ are homotopic $A_\infty$-morphisms (see \cite[1.2.1.7]{LH}), then $g_2(\alpha\otimes\alpha)=g_2'(\alpha\otimes\alpha)$.
\end{lem}

\begin{proof}
    If $g'$ is homotopic to $g$, then there is a map $h_1\colon P_n A \rightarrow P_n B$ of degree $-1$ satisfying
    \[(g-g')_2=h_1\circ m_2 - m_2\circ (g_1 \otimes h_1) - m_2\circ (h_1 \otimes g'_1). \]
    But $m_2(\alpha\otimes\alpha)=0$ and $h_1(\alpha)=0$ since $(P_n B)^1=0$.
\end{proof}

 If $M_1\rightarrow P_n A$ and $M_2\rightarrow P_n B$ are minimal cdga models, then by the argument in the last paragraph of the proof of \Cref{formalmap} the strictly unital $C_\infty$-morphism $f$ lifts to a cdga morphism $M_f\colon M_1\rightarrow M_2$ (i.e. the resulting square in the homotopy category of $C_\infty$-algebras commutes).

\begin{cor}\label{nonzerodegnonformal}
    For $n\geq 8$, the non-zero degree map $M_f$ is a dominant non-formal map of cdga's.
\end{cor}

\begin{proof}
We observe that $P_nA$ and $P_nB$ are formal minimal $A_\infty$-algebras and so we identify them with their cohomology.
Assuming $M_f$ to be formal, we obtain a commutative diagram in the homotopy category of $C_\infty$-algebras.

\[\begin{tikzcd}
P_nA \arrow[d,  "f"] & M_1 \arrow[r, "\sim"] \arrow[d, "M_f"] \arrow[l, "\sim"'] & P_nA \arrow[d, "H(f)"] \\
P_nB                & M_2 \arrow[r, "\sim"] \arrow[l, "\sim"']                  & P_nB                  
\end{tikzcd}\]
where the right hand square arises from applying criterion \ref{itf: Cinfty} of \Cref{formalmap} to the map of cdga's $M_f$.

It is now sufficient to consider this diagram in the category of (non-unital)  $A_\infty$-algebras. Recall that any quasi-isomorphism admits a quasi-inverse and that furthermore any quasi-isomorphism between minimal $A_\infty$-algebras is an isomorphism. Thus, the previous diagram gives rise to the following diagram of $A_\infty$-algebras 
    \[\begin{tikzcd}
P_n A \arrow[d, "f"'] & \arrow[l, "\varphi"] P_n A  \arrow[d, "H(f)"]  \\
P_nB  \arrow[r, "\psi"]                                   & P_n B 
\end{tikzcd}\]
     whose image in the homotopy category of $A_\infty$-algebras commutes and in which $\varphi,\psi$ are isomorphisms.
Hence we find a homotopy between $\psi\circ f\circ\varphi$ and $H(f)$. Since the higher components of $H(f)$ are trivial this cannot happen by \Cref{prop:automorphism2ary} and \Cref{lem:homotopy2ary}.\end{proof}

We say a map of spaces $Y \to X$ is formal if the induced map $A_{PL}(X) \to A_{PL}(Y)$ is formal in the above sense. The geometric realization of $M_f$ thus gives the desired example for \Cref{thmC}.

Geometrically, the map $f\colon A\to B$ we started with corresponds to the map $S^2 \vee (S^3 \times S^4) \to S^2$ given by the identity on the $S^2$ summand, and by the projection to $S^3$ followed by the Hopf map $h\colon S^3 \to S^2$ on the $S^3 \times S^4$ summand. This map factors through $S^2 \vee S^3$ and in fact one can show that $\Id_{S^2}\vee h\colon S^2\vee S^3\rightarrow S^2$ is formal\footnote{As we are seeing here, generally the composition of two formal maps need not be formal; see e.g. \cite[p.106, Proposition]{FT88}.}: indeed, it is homotopic to $p\circ \psi$ where $p=\Id_{S^2}\vee *\colon S^2\vee S^3\rightarrow S^2$ is formal and $\psi=\Id_{S^2}\vee (h+\Id_{S^3})\colon S^2\vee S^3\rightarrow S^2\vee S^3$ is a self-equivalence, where $h+\id_{S^3}$ denotes the sum in $\pi_3(S^2\vee S^3)$. The role of the $S^4$ factor in the second summand above is to restrict the possible automorphisms intertwining $S^2$ and $S^3$ in order to ensure non-formality.

\section{Geometric interpretation of the dualization}\label{universalpropertyandgeometricinterpretation}

We come now to the geometric interpretation of the algebraic dualization construction. This will extend the result of \cite[Theorem 6]{La00} which constructs a model for the boundary of a stable thickening of a simply connected finite cell complex $X$ using an appropriately truncated model\footnote{Note that the double of a stable thickening of $X$ is homotopy equivalent to the boundary, upon smoothing corners, of $[0,1]$ crossed with that stable thickening.} of $X$.

Let $X$ be any finite cell complex. Take an $n$-dimensional orientable manifold-with-boundary $M$ homotopy equivalent to $X$ (e.g. $M$ can be obtained by embedding $X$ into some $\mathbb{R}^n$ and thickening it to a manifold with boundary). Denote by $D(M)$ the double of $M$. Since the inclusion of $M\to D(M)$ is a one-sided inverse to the projection $D(M)\to M$, the long exact sequence of the pair $(D(M),M)$ splits and we obtain an additive decomposition
\[
H^\ast (D(M))\cong H^*(M)\oplus H^*(D(M),M)\cong H^*(M) \oplus H^*(M, \partial M) \cong H^*(X)\oplus D_n H^*(X)
\]
where the second isomorphism follows by excision, and the last follows by Poincar\'e--Lefschetz duality. This is also true multiplicatively, i.e. $H^*(D(M))\cong P_nH^*(X)$ as algebras. In \Cref{prop:topological interpretation} below, we prove the (more general) version of this statement on the level of cdga's.

\begin{lem}\label{7.1} Let $A \xrightarrow{f} B$ be a map of cohomologically finite type cdgas, where the cohomology of $B$ furthermore satisfies Poincar\'e duality. Then if $f$ admits a retract $B \xrightarrow{r} A$ of $A$-modules, it extends to an $A_\infty$-morphism $P_n A \to B$, where $n$ is the cohomological dimension of $B$. \end{lem}

\begin{proof} Poincar\'e duality of $B$ provides a quasi-isomorphism of dg-$B$-modules $B \to D_n B$, given by $b\mapsto b\wedge F$, where $F$ is a representative of a fundamental class (thought of as a degree zero element in $D_n B$). In particular, this is a quasi-isomorphism of dg-$A$-modules, and hence we may find a quasi-inverse $\psi$ in the category of $A_\infty$-$A$-bimodules, see \Cref{lem: quasi inverse}. The composition $\psi \circ D_n r$ is a map of $A_\infty$-$A$-bimodules, and hence we can apply \Cref{mapfromPD}. \end{proof}

\begin{prop}\label{prop:topological interpretation}(cf. \cite[Theorem 6]{La00})
    Let $A$ be a connected cohomologically finite type cdga and $X$ a finite cell complex such that $A_{PL}(X)$ is weakly equivalent to $A$. Take an $n$-dimensional orientable manifold-with-boundary $M$ homotopy equivalent to $X$ (e.g. $M$ can be obtained by embedding $X$ into some $\mathbb{R}^n$ and thickening it to a manifold with boundary). Then $P_{n}A$ is weakly equivalent to $A_{PL}(D(M))$ where $D(M)$ denotes the double of $M$.
\end{prop}

\begin{proof}
    Since $X$ and $M$ are homotopy equivalent, we may replace $A_{PL}(X)$ by the quasi-isomorphic $A_{PL}(M)$. Consider the obvious retract for the inclusion of $M$ into its double $M \hookrightarrow D(M)$. On piecewise-linear forms we thus have a map $A_{PL}(M) \to A_{PL}(D(M))$ which admits a retract. We thus obtain by \Cref{7.1} an $A_\infty$-morphism $P_n(A_{PL}(M)) \to A_{PL}(D(M))$. It is cohomologically injective as a non-zero degree map between cdga's that satisfy cohomological Poincaré duality. Furthermore, this map is an $A_\infty$-quasi-isomorphism since the sum of Betti numbers of $P_n(A_{PL}(M))$ is twice that of $M$, and the same holds for $A_{PL}(D(M))$ by the initial discussion in this section. By \cite{CPRNW19}, $P_n(A_{PL}(M))$ and $A_{PL}(DM)$ are also weakly equivalent as cdgas.
\end{proof}



\section{Simple Poincar\'e duality models}\label{simplePDmodels}
In this section, we revisit the theory of cyclic models for Poincar\'e duality spaces to recover, and extend, results of the form `connectivity implies formality'. 
\begin{defi}
Let $(A,m_*)$ be a minimal $A_\infty$-algebra with commutative $m_2$ satisfying Poincaré duality on its cohomology with dimension $n$. We say $A$ is \emph{simple} if $m_k$ maps trivially to $A^n$ for $k\geq 3$. \end{defi}

This definition is motivated by the observation that on a cdga satisfying Poincar\'e duality on its cohomology, the triple and higher Massey products landing in top degree are trivial. Indeed, given a Massey product $\langle x_1, \ldots, x_k \rangle$, one can perturb any chosen primitive of $\langle x_2, \ldots, x_k \rangle$ by any class pairing non-trivially with $x_1$ and hence scale the output of the $k$-fold Massey product (which lies in the one-dimensional top cohomology) to zero; this argument is from \cite[after Lemma 7]{CFM08}.

\begin{ex}
Let $(A,m_*)$ be any strictly unital minimal $A_\infty$-algebra (i.e.\ the $m_k$, $k\geq 3$ vanish whenever one plugs in $1$ as one of the inputs; such models always exist for unital dgas \cite[\textsection 3.2.1]{LH}) with commutative $m_2$, and $P_N(A)$ its Poincaré dualization. Then $P_N(A)$ is simple. Indeed for $x_i\in A$, $\varphi\in D_N(A)$, evaluating $m_k(x_1,\ldots,x_{l-1},\varphi,x_{l+1},\ldots,x_k)$ on $1\in A^0$ we obtain $$\varphi(m_k(x_{l+1},\ldots,1,\ldots, x_{l-1}))=0.$$
\end{ex}

These models are useful as the operadic Massey products can be computed using the same formula as was used for the definition of Poincar\'e dualization,  \cref{formulaformi}:

\begin{prop}\label{prop:formulainsimplyPDmodel}
    Let $(A,m_*)$ be simple, and let $\Phi\colon A^k\rightarrow (A^{N-k})^{\vee}$ denote the isomorphism induced by a choice of fundamental class and Poincaré duality, i.e. the map defined via the homomorphism $\int\colon A^N\rightarrow \mathbb{Q}$ so that
$$\left(\Phi(a)\right)(b)=  \int m_2(a,b).$$ Then, for any $l\leq k$ we have
$$\left(\Phi(m_k(x_1,\ldots,x_k))\right) (b) = (-1)^{kl + \Sigma_1^{l-1}(\Sigma_{l+1}^k + |b| + |x_l|) + k|x_l|} \Phi(x_l) (m_k(x_{l+1}, \ldots,x_k, b, x_1, \ldots, x_{l-1})),$$ where $\Sigma_i^j$ denotes $\sum_{t=i}^j |x_t|$. \end{prop}

\begin{proof}
The left-hand side of the above equation equals $\int m_2(m_k(x_1, \ldots, x_k), b)$. On the other hand we have, using the commutativity of $m_2$, that 
\begin{align*} \Phi(x_l)&(m_k(x_{l+1},\ldots, x_k,b,x_1,\ldots,x_{l-1}) = \int m_2(x_l,m_k(x_{l+1},\ldots,b,\ldots, x_{l-1})) \\ &= (-1)^{|x_l|(\Sigma_1^{l-1} + |b| + \Sigma_{l+1}^k + k)} \int m_2(m_k(x_{l+1}, \ldots, x_k, b, x_1, \ldots, x_{l-1}), x_l). \end{align*}

Let us denote $x_{k+1} = b$ when it simplifies the notation. To finish the proof, let us consider the left-hand side again, and observe that the arguments $x_1,\ldots,x_k,b$ of $m_2\circ (m_k \otimes 1)$ may be permuted cyclically, up to a change in sign, in case their degrees add up to $N+k-2$.  To see this observe first that by the $A_\infty$-equation and the fact that $A$ is a simple Poincar\'e duality model we get $$ 0 = -m_2(1 \otimes m_k) + (-1)^k m_2(m_k \otimes 1)$$ when evaluated on $(x_1, \ldots, x_{k+1})$. Evaluating and using that $m_2$ is commutative, this becomes $$m_2(m_k(x_1, \ldots, x_k), x_{k+1}) = (-1)^{k + |x_1|\Sigma_2^{k+1}} m_2(m_k(x_2, \ldots, x_{k+1}), x_1).$$

Iterating $l$ times, we obtain $$m_2(m_k(x_1, \ldots, x_k), x_{k+1}) = (-1)^{lk + \Sigma_1^l(\Sigma_1^{k+1} + 1)} m_2(m_k(x_{l+1}, \ldots, x_k, x_{k+1}, x_1, \ldots, x_{l-1}), x_l).$$ From here we have $$\left(\Phi(m_k(x_1,\ldots,x_k))\right) (b) = (-1)^{k(l+|x_l|) + \Sigma_1^{l-1}(\Sigma_1^k + |b| + 1)} \left( \Phi(x_l) \right) (m_k(x_{l+1}, \ldots,x_k, b, x_1, \ldots, x_{l-1})),$$ which is equivalent to the desired equation. \end{proof}

The formula in the statement of \Cref{prop:formulainsimplyPDmodel} is equivalent to the following: for any $i+1+j=k$ the diagram
\[\xymatrix{A^{\otimes k}\ar[rr]^(.4){1^{\otimes i}\otimes \Phi\otimes 1^{\otimes j}}\ar[d]^{m_k}& & A^{i} \otimes A^\vee\otimes A^{\otimes j}\ar[d]\\ A\ar[rr]^{\Phi}& & A^\vee }
\]
commutes, where the right hand vertical map is the $A_\infty$-$A$-bimodule structure on $A^\vee$ from \cref{formulaformi}. I.e., $\Phi$ defines a strict morphism of $A_\infty$-$A$-bimodules $A\to A^\vee$.

This says that $A$ together with the pairing given by $\Phi$ is a \textbf{cyclic} $A_\infty$-algebra, see e.g. \cite[Definition 3.1]{CL11}, and \cite{HL08} for the $C_\infty$ case, where the term \emph{symplectic} is used instead of \emph{cyclic}; for motivation of the latter terminology we refer the reader to the original \cite{K93}.

On the other hand, it is clear that a unital minimal cyclic $A_\infty$-algebra with commutative $m_2$ is simple. Namely, for $m_k(x_1, \ldots, x_k)$ in top degree, we have $$\int m_k(x_1, \ldots, x_k) = \Phi(1)m_k(x_1, \ldots, x_k) = \pm \int m_2(x_1, m_k(x_2, \ldots, x_k, 1)),$$ which vanishes by unitality.

\begin{cor}\label{simpleiscyclic} A strictly unital minimal $A_\infty$-algebra with commutative $m_2$ satisfying Poincar\'e duality on its cohomology is simple if and only if, upon making a choice of fundamental class, it is cyclic with respect to the pairing given by $\Phi$ defined above. \end{cor} 

The existence of cyclic (also known as symplectic or Frobenius) $C_\infty$-models, or $A_\infty$-models, is well known; see \cite[Theorem 5.5]{HL08}, \cite[1.2]{CL09}, \cite[Corollary 29]{KTV21}, \cite[Theorem 10.2.2]{KS06}.

As applications, we now prove an analogue of \cite[Theorem 3.1]{FM05}, that a cdga satisfying Poincar\'e duality on its cohomology, of dimension $N$, is formal if and only if it is $\lceil \frac{N}{2} - 1 \rceil$--formal in the sense of loc. cit., and recover (and extend) some known formality results for Poincar\'e duality spaces.

\begin{cor}
Let $A$ be a simple, strictly unital, $A_\infty$-algebra in which $m_k$ vanishes on $(A^{\leq \lceil N/2 - 1 \rceil })^{\otimes_k}$ for $k\geq 3$. Then $A$ is formal.
\end{cor}

\begin{proof}
Let $x_1,\ldots,x_k\in A$ and assume $|x_l| \geq N/2$ for some $1 \leq l \leq k$. For degree reasons, we can assume all other inputs are of degree $\leq N/2$. In fact, if $|x_l| > N/2$, we can assume all other inputs are of degree $< N/2$. If $|x_l| = N/2$, there might be another input of degree $N/2$, in which case, in order for the output to be in degree $\leq N$, all $k-2$ other inputs must be of degree 1; in this case the output is in degree $N$ and hence $m_k$ vanishes by simplicity. Therefore we assume we have a single $x_l$ of degree $\geq N/2$, and for the other inputs, $|x_i| < N/2$. 

To show $m_k(x_1, \ldots, x_k)$ vanishes, we show that $\Phi(m_k(x_1,\ldots, x_k))$ vanishes when evaluated on any $b$, whose degree we may assume is $N-|m_k(x_1, \ldots, x_k)| < N/2$. Now, $$\Phi(m_k(x_1,\ldots, x_k))(b) = \pm \Phi(x_l)(m_k(x_{l+1}, \ldots, x_k, b, x_1, \ldots, x_{l-1})),$$ and the input on the right-hand side vanishes by assumption.
\end{proof}

\begin{rem} It is not true that $s$-formality in the sense of \cite{FM05} is equivalent to the existence of a minimal $C_\infty$ model with $m_{\geq 3}$ vanishing on inputs whose individual degrees are $\leq s$. Indeed, consider for example the cdga $$A=\left( \Lambda(a,b,c,x,y), da = 0, db = 0, dc = ab, dx = a^5, dy = b^5\right),$$ where $\deg(a) = \deg(b) = 2, \deg(c) = 3, \deg(x) = \deg(y) = 9$. Its cohomology satisfies Poincar\'e duality with formal dimension 19, a basis is given by \begin{align*} &[1], [a], [b], [a^2], [b^2], [a^3], [b^3], [a^4], [b^4],
[b^4 c -ay], [bx -a^4 c],
[a b^4 c -a^2 y], [b^2 x - a^4 b c], \\
&[a^2 b^4 c-a^3 y], [b^3 x-a^4 b^2 c],
[a^3 b^4 c-a^4 y], [b^4 x-a^4 b^3 c ],
[ a^4 b^4 c-a^5 y].
\end{align*}

The cdga $A$ is easily seen to be 8-formal in the sense of \cite{FM05}: For this one needs to check that every closed element in $c\cdot\Lambda(a,b,c)$ is exact in $A$. But, the only closed element here is $0$. On the other hand, $A$ is not formal. For instance, it has the nontrivial Massey product $\langle [a^4],[a],[b]\rangle =\{[bx-a^4c]\}\not\ni\{0\}$.

Finally, if we were to find a minimal $A_\infty$ model with $m_{\geq 3}$ vanishing on inputs of degree $\leq 8$, then $m_{\geq 3}$ would also vanish on inputs of degree $\leq 9$ since $H^9$ is trivial, implying formality. 

\end{rem}

\begin{rem} We note how the existence of minimal strictly unital cyclic $C_\infty$-algebra models recovers Miller's theorem \cite{M79} that a $k$--connected rational Poincar\'e duality space of formal dimension $\leq 4k+2$ is formal. Namely, choosing a minimal cyclic $C_\infty$-model, we see that its higher operations $m_{\geq 3}$ necessarily vanish: for degree reasons the only possibly non-trivial output would lie in top degree, which is excluded by simplicity. 

Furthermore, we can recover part of the extension of Miller's theorem by Cavalcanti \cite[Theorem 1]{Ca06} that formality still holds if the formal dimension is $\leq 4k+4$ and $b_{k+1} = 1$. Namely, one need only additionally consider higher products of the form $m_3(x,x,x)$ for $x \in H^{k+1}$, and, in the case of dimension $4k+4$, those of the form $m_3(x,x,y)$ for $y \in H^{k+2}$ (or permutations thereof). The products $m_3(x,x,x)$ vanish since we are in a $C_\infty$-algebra and $m_3$ vanishes on shuffles; note that there are three $(1,2)$ shuffles, for example. In dimensions of the form $4k+4$, by Poincar\'e duality a product of the form $m_3(x,x,y) \in H^{3k+3}$ vanishes if and only if $m_2(x, m_3(x,x,y)) = 0$, which holds since by cyclicity $m_2(x, m_3(x,x,y)) = \pm m_2(y, m_3(x,x,x)) = 0$. 

In \cite[Conjecture 1]{Zh19}, Zhou conjectured a generalization of the above, in the following form: an $n$--dimensional $k$-connected closed manifold
with $b_{k+1} = 1$ admits a minimal $A_\infty$ model with $m_{\geq j} = 0$ for any $j \geq 3$ such that $n \leq (j+1)k + 4$. We can confirm this immediately for all $j \geq 3$ with $n \leq (j+1)k + 2$, by taking a minimal simple $C_\infty$ model. If $k \geq 2$ and $n = (j+1)k + 3$, or $k \geq 2$ and $n = (j+1)k + 4$, we can verify the conjecture if $j$ is not a power of two. Indeed, one sees that $m_j(x, x, \ldots, x)$ for $x \in H^{k+1}$ vanishes since there is some $0 < i < j$ such that the $(i,j-i)$ shuffle equation has an odd number of terms; i.e. since $j$ is not a power of two, some binomial coefficient $\binom{j}{i}$ is odd. The only other potentially nonvanishing $m_{\geq j}$ are of the form $m_j(y,x,x, \ldots, x)$ and permutations thereof, where $y \in H^{k+2}$; these vanish again by cyclicity. In the case $k=1$ and $n = (j+1)k + 3, (j+1)k + 4$, we can verify the conjecture if we furthermore assume that $j+1$ is also not a power of two, to ensure the vanishing of $m_{j+1}(x,x,\ldots, x)$. 

Since the above arguments only make use of the vanishing of the cohomology in certain degrees and Poincar\'e duality, we are furthermore concluding these underlying cohomology algebras are \emph{intrinsically formal}. \end{rem}

\end{document}